\documentclass[11pt,leqno]{amsart}

\usepackage[top=2.5 cm,bottom=2 cm,left=2.5 cm,right=2 cm]{geometry}
\usepackage[utf8]{inputenc}
\usepackage{mathrsfs}
\usepackage{bbm}
\usepackage{esint}
\usepackage{graphicx}
\usepackage{hyperref}\hypersetup{colorlinks=true, citecolor=blue}

\usepackage{color}
\usepackage{amssymb}
\usepackage{amsmath,amscd}
\usepackage{mathtools}
\mathtoolsset{showonlyrefs,showmanualtags}

\newcounter{stepnb}

\newtheorem{theorem}{Theorem}[section]
\newtheorem{lemma}[theorem]{Lemma}
\newtheorem{counterexample}[theorem]{Counterexample}
\newtheorem{proposition}[theorem]{Proposition}

\theoremstyle{definition}
\newtheorem{remark}[theorem]{Remark}
\theoremstyle{definition}
\newtheorem{question}{Question}

\numberwithin{equation}{section}

\newcommand{\dd}{\mathsf d}
\newcommand{\R}{\mathbb{R}}

\newcommand{\ee}{\varepsilon}

\newcommand{\eps}{\varepsilon}
\renewcommand{\div}{{\rm div}\,}

\newcommand{\weaks}{\stackrel{*}{\rightharpoonup}}
\newcommand{\loc}{\mathrm{loc}}

\begin{document}

\title[Local limit of nonlocal conservation laws]{On the singular local limit for conservation laws\\with nonlocal fluxes}

\author[M.~Colombo]{Maria Colombo}
\address{M.C. EPFL, Station 8, 
CH-1015 Lausanne, Switzerland.}
\email{maria.colombo@epfl.ch}
\author[G.~Crippa]{Gianluca Crippa}
\address{G.C. Departement Mathematik und Informatik,
Universit\"at Basel, Spiegelgasse 1, CH-4051 Basel, Switzerland.}
\email{gianluca.crippa@unibas.ch}
\author[L.~V.~Spinolo]{Laura V.~Spinolo}
\address{L.V.S. IMATI-CNR, via Ferrata 5, I-27100 Pavia, Italy.}
\email{spinolo@imati.cnr.it}
\maketitle
{
\rightskip .85 cm
\leftskip .85 cm
\parindent 0 pt
\begin{footnotesize}

{\sc Abstract.}
We give an answer to a question posed in~\cite{ACT}, which can be loosely speaking formulated as follows. Consider a family of continuity equations where the velocity depends on the solution via the convolution by a regular kernel. In the singular limit where the convolution kernel is replaced by a Dirac delta, one formally recovers a conservation law: can we rigorously justify this formal limit? 
We exhibit counterexamples showing that, despite numerical evidence suggesting a positive answer, one in general does not have convergence of the solutions. We also show that the answer is positive if we consider viscous perturbations of  the nonlocal equations. In this case, in the singular local limit the solutions converge to the solution of the viscous conservation law. 
%\emph{Draft, \today}

\medskip\noindent
{\sc Keywords:} nonlocal conservation law, nonlocal continuity equation, singular limit, local limit. 

\medskip\noindent
{\sc MSC (2010):} 35L65.

\end{footnotesize}

}

\vspace{.3 cm}

\section{Introduction and main results}
We are concerned with the so-called nonlocal continuity equation (or nonlocal conservation law)
\begin{equation}
\label{e:intro}
    \partial_t w + \div [ w \ b (w \ast \eta)] =0. 
\end{equation}
In the previous expression, $b: \R \to \R^d$ is a Lipschitz continuous vector-valued function, the scalar function $w: \R^+\times\R^d\to\R$ is the unknown and $\div$ denotes the divergence computed with respect to the space variable only. The symbol $\ast$ denotes the convolution computed with respect to the space variable only and $\eta$ is a convolution kernel  satisfying 
\begin{equation}
\label{e:eta}
      \eta: \R^d \to \R, \qquad \eta \in C_c^\infty (\R^d), \quad \eta (x) = 0 \; \text{if $|x| \ge 1$}  , \quad \eta \ge 0, \quad \int_{\R^d} \eta (x) dx =1.
\end{equation} 
In recent years, conservation laws involving nonlocal terms have been extensively studied owing to their applications to models for sedimentation~\cite{Sedimentation}, pedestrian~\cite{ColomboGaravelloMercier} and vehicular~\cite{BlandinGoatin} traffic, and others. 
We refer to the recent paper~\cite{BlandinGoatin} for a more extended discussion and a more complete list of references. Here we only mention that the basic idea underpinning    
the use of equations like~\eqref{e:intro} in traffic models is, very loosely speaking, the following. The unknown $w$ represents the density of pedestrians or cars and $b$ their velocity. The nonlocal term $w \ast \eta$ appears since one postulates that pedestrians or drivers tune their velocity depending on the density of pedestrians or cars surrounding them. %{\color{blue}Citare libro Garavello Han Piccoli? A Pavia non lo abbiamo, potete controllare voi se parlano di modelli nonlocali?}. 

In the present work we investigate a question posed by Amorim, R. Colombo and Teixeira in~\cite{ACT}. To precisely state the question, we consider the family of Cauchy problems  
\begin{equation}
\label{e:nlcpr}
\left\{
\begin{array}{ll}
  \partial_t u_{\ee } + \div \big[ u_{\ee } b(u_{\ee } \ast \eta_\ee) \big] = 0 \\
  u_{\ee } (0, x) = \bar u(x), \\
\end{array}
\right.
\end{equation}
where $b$ is as before a Lipschitz continuous vector-valued function, $\ee$ is a positive parameter and $\bar u$ is a summable and bounded initial datum. Assume that the family of convolution kernels $\eta_\ee$ is obtained from $\eta$ by setting 
\begin{equation}
\label{e:etaee}
    \eta_\ee (x) : = \frac{1}{\ee^d} \eta \left( \frac{x}{\ee} \right), \qquad 0 < \varepsilon \leq 1,
\end{equation}
in such a way that when $\ee \to 0^+$ the family $\eta_\ee$ converges weakly-$^\ast$ in the sense of measures to the Dirac delta. This implies that, when $\ee \to 0^+$, the Cauchy problem~\eqref{e:nlcpr} \emph{formally} reduces to a scalar conservation law
\begin{equation}
\label{e:clcpr}
\left\{
\begin{array}{ll}
  \partial_t u + \div \big[ u b(u) \big] = 0 \\
  u (0, x) = \bar u(x). \\
\end{array}
\right.
\end{equation}
The by now classical theory by Kru{\v{z}}kov~\cite{Kruzkov} provides global existence and uniqueness results for so-called \emph{entropy admissible} solutions. We refer to~\cite{Dafermos:book}
for the definition and an extended discussion concerning entropy solution of conservation laws. 
The question posed in~\cite{ACT} can be formulated as follows. 
 \begin{question}
 \label{?} Can we rigorously justify the singular 
 limit from~\eqref{e:nlcpr} to~\eqref{e:clcpr}? In other words, does $u_\ee$ converge to the entropy admissible $u$, in a suitable topology? 
  \end{question}
  Some remarks are here in order. First, Question~\ref{?} is motivated by numerical experiments. Indeed, in~\cite[\S~3.3]{ACT} the authors exhibit numerical evidence suggesting that there should be convergence. Second, to the best of our knowledge, the only previous analytical result concerning Question~\ref{?} is due to Zumbrun~\cite{Zumbrun} and states that the answer to Question~\ref{?} is positive provided that the limit entropy solution $u$ is smooth  and the convolution kernel is even, i.e. $\eta (x) = \eta (-x)$ (see~\cite[Proposition 4.1]{Zumbrun} for a more precise statement). Third, even in the case $d=1$, $b(u_\ee) =u_\ee$, establishing  weak compactness of the family $\{ u_\ee \}$ is not a priori sufficient to establish convergence. Indeed, one needs strong convergence (or some more refined argument) to pass to the limit in the nonlinear term $u_\ee \ u_\ee \ast \eta_\ee.$  Fourth, similar questions show up when considering equations in transport form instead of in continuity form as in~\eqref{e:nlcpr} and~\eqref{e:clcpr} (see for instance~\cite{Fetecau}); the analysis in such a case shares some similarities with that in the present paper and we plan to address it in future work. 

In this paper we exhibit explicit counterexamples 
showing that the answer to Question~\ref{?} is, in general, negative. Also, we show that the answer is positive if we add to the right hand side of the first line of both~\eqref{e:nlcpr} and~\eqref{e:clcpr} a viscous term. 
As we explain below, this is relevant in connection with the numerical analysis of the singular limit 
from~\eqref{e:nlcpr} to~\eqref{e:clcpr}. 

We now describe our results more precisely. Our counterexamples can be summarized as follows:
\begin{itemize}
\item In~\S~\ref{ss:ce1} we exhibit a counterexample showing that, in general, $u_\ee$ does not converge to the entropy admissible solution $u$ weakly in $L^p$ or weakly$^\ast$ in $L^\infty$.  The example uses a family of even convolution kernels and is described in Counterexample~\ref{l:ce1}. A drawback is that the initial datum $\bar u$ changes sign. This is not completely satisfactory in view of the applications, where the unknown typically represents a density. 
\item In~\S~\ref{ss:ce2} we exhibit a counterexample with a nonnegative initial datum where we show that $u_\ee$ does not converge to $u$ weakly in $L^p$ or weakly$^\ast$ in $L^\infty$. See Counterexample~\ref{l:ce2} for the precise statement. A drawback of this example is that we have to use ``completely asymmetric" convolution kernels, namely we assume that $\eta(x) =0$ for every $x>0$.  Note that this is consistent with numerical experiments provided in~\cite[\S~3.2]{ACT} and~\cite[\S~5]{BlandinGoatin}, where ``completely asymmetric" kernels are connected with highly oscillatory behaviors of the solution. 
\item In \S~\ref{ss:ce3} we exhibit a counterexample involving a nonnegative initial datum and a family of even convolution kernels.  In this counterexample we show that for every $\delta>0$ the family $u_\ee$ does not converge to $u$ strongly in $L^{1+ \delta}$. See Counterexample~\ref{l:ce3} for the precise statement. 
\end{itemize} 
To find a contradiction to the convergence, in the three examples we construct a family of solutions $u_\ee$ with some qualitative property which is stable under convergence, but not satisfied by the entropic solution in the limit. 
These qualitative properties differ in each case and are, roughly speaking, related to the total mass of the solution in a suitable region (see Counterexample~\ref{l:ce1}), the support (see Counterexample~\ref{l:ce2}), and the quantity $\int u \log u \, dx$ (which, under suitable assumptions, is conserved by the nonlocal approximation and strictly dissipated in the limit, see Counterexample~\ref{l:ce3}).
A more precise description of the idea behind each counterexample can be found after each statement in Section~\ref{s:ce}.

%{\color{blue}Ci sarebbe qui la questione che se il nucleo di convoluzione \`e sbilanciato dalla parte ``giusta" ci potrebbe essere convergenza. Volete menzionare la cosa? Io sarei per il no, ditemi per\`o anche voi.}
As mentioned before, we manage to establish positive results by adding to the first line of~\eqref{e:nlcpr} and~\eqref{e:clcpr} a second order perturbation. More precisely, we consider the family of Cauchy problems 
\begin{equation}
\label{e:cpr}
\left\{
\begin{array}{ll}
  \partial_t u_{\ee \nu} + \div \big[ u_{\ee \nu} b(u_{\ee \nu} \ast \eta_\ee) \big] = \nu \Delta u_{\ee \nu} \\
  u_{\ee \nu} (0, x) = \bar u(x), \\
\end{array}
\right.
\end{equation}
which depends on  two parameters $\ee >0$ and $\nu >0$. When $\ee \to 0^+$ and $\nu$ is fixed, the family of Cauchy problems~\eqref{e:cpr} \emph{formally} reduces to 
\begin{equation}
\label{e:vclr}
\left\{
\begin{array}{ll}
\partial_t u_\nu + \div \big[ u_\nu b (u_\nu) \big] = \nu \Delta u_\nu \\
u_\nu (0, x) = \bar u (x).  \\
\end{array}
\right.
\end{equation}
On the other hand, when $\nu \to 0^+$ and $\ee$ is fixed, the family of Cauchy problems formally reduces to~\eqref{e:nlcpr}, while~\eqref{e:vclr} reduces to~\eqref{e:clcpr} (see~\eqref{e:disegno} below for a scheme). The reason why we consider the viscous approximations~\eqref{e:cpr}, \eqref{e:vclr} is the following. As mentioned before, Question~\ref{?} is motivated by the numerical evidence exhibited in~\cite{ACT}. The numerical tests showing convergence are obtained by using a Lax-Friedrichs type scheme involving some so-called \emph{numerical viscosity}, as it typical of many numerical schemes for conservation laws (see the book by LeVeque~\cite{LeVeque} for an extended introduction). Very loosely speaking, the numerical viscosity consists of finite differences terms that mimic a second order operator like the Laplacian. For this reason, the analysis of the viscous approximation~\eqref{e:cpr},~\eqref{e:vclr}  may provide some insight in the understanding of the numerical tests. See also~\cite{ColomboCrippaKraySpinolo:numer} for further numerical investigations.

Our main result involving the singular limit from~\eqref{e:cpr} to \eqref{e:vclr} is the following. 
\begin{theorem}
\label{t:convbadata}
Let $b$ be a Lipschitz continuous function, $\bar u \in L^1 (\R^d) \cap L^\infty (\R^d)$, $\nu >0$ and $p$ such that
\begin{equation}
\label{e:condizionip}
    2 \leq  p<  \infty, \; p > d. 
\end{equation}
Let $u_{\ee \nu}$ and $u_\nu$ be the solutions of~\eqref{e:cpr} and~\eqref{e:vclr} starting from $\bar u$, respectively. Then
$$u_{\ee \nu} \to u_\nu \qquad \mbox{strongly in }
L^{\infty}_{\mathrm{loc}}([0, + \infty[; L^p(\R^d)).
$$\end{theorem}
Some remarks are here in order. First, the Cauchy problem~\eqref{e:cpr} has a unique weak solution, see Theorem~\ref{t:wp} in~\S~\ref{s:preliminary} for the precise statement. Second, in the case $d=1$, $b(u) =u$, $p=2$, $\bar u \in W^{1, \infty} (\R)$, Theorem~\ref{t:convbadata} was established by Calderoni and Pulvirenti~\cite{CalderoniPulvirenti}. The main novelties of Theorem~\ref{t:convbadata} with respect to the analysis in~\cite{CalderoniPulvirenti} can be summarized as follows: 
\begin{itemize}
\item We provide a completely different proof. Indeed, in~\cite{CalderoniPulvirenti} the authors explicitly compute the equations satisfied by the Fourier transforms $\hat{u}_{\ee \nu}$ and $\hat{u}_\nu$ and use them to control the $L^2$ norm of the difference. The proof explicitly uses the fact that $b(u)=u$ and the regularity of the initial datum.  
\item On the other hand, our argument is based on a-priori estimates obtained by extensively using energy estimates and the Duhamel representation formula. We first establish Theorem~\ref{t:convreg} in the case when the initial datum $\bar u$ is regular. Next, we introduce a careful perturbation argument and we establish the proof in the general case. Our argument is fairly robust, it applies to general functions $b$, to equations in several space dimensions, and to rough initial data, and provides more quantitative estimates, see Remark~\ref{p:diagonal} below. 
\end{itemize}
As a further remark, we explicitly point out that Theorem~\ref{t:convbadata} requires neither symmetry conditions on the convolution kernels $\eta_\ee$ nor sign conditions on the initial datum $\bar u$. 

Finally, we discuss the vanishing viscosity limit from~\eqref{e:cpr} to~\eqref{e:nlcpr}. Our result is the following. 
\begin{proposition}
\label{p:vnltonl}
Under the assumptions of Theorem~\ref{t:convbadata}, let $u_{\ee \nu}$ and $u_{\ee}$ satisfy~\eqref{e:cpr} and~\eqref{e:nlcpr}, respectively. For every $\ee >0$, we have that $u_{\ee \nu} \weaks u_\ee$ weakly$^\ast$ in $L^\infty_\mathrm{loc} ([0, + \infty[ \times \R^d)$ as $\nu \to 0^+$.
\end{proposition}
Note that in the statement of Proposition~\ref{p:vnltonl} the parameter $\ee >0$ is fixed and  hence the weak$^\ast$ convergence suffices to pass to the limit in the equation, owing to the regularizing effect of the convolution. Also, note that at the local level the vanishing viscosity limit from~\eqref{e:vclr} to~\eqref{e:clcpr} is established in the work by Kru{\v{z}}kov~\cite{Kruzkov}. 

The take-home message obtained by combining the counterexamples, Theorem~\ref{t:convbadata}, Kru{\v{z}}kov's Theorem and Proposition~\ref{p:vnltonl} can be therefore represented as follows: 
\begin{equation}
\label{e:disegno}
\minCDarrowwidth70pt
\begin{CD}
 \partial_t u_{\ee \nu} + \div \big[ u_{\ee \nu} b(u_{\ee \nu} \ast \eta_\ee) \big] = \nu \Delta u_{\ee \nu}     @> \ee \to 0^+  >  \text{Theorem~\ref{t:convbadata} }>  \partial_t u_\nu + \div \big[ u_\nu b (u_\nu) \big] = \nu \Delta u_\nu \\
@V \nu \to 0^+  V \text{Proposition~\ref{p:vnltonl}} 
V        @V \nu \to 0^+ V \text{Kru{\v{z}}kov's Theorem} V\\
 \partial_t u_{\ee } + \div \big[ u_{\ee } b(u_{\ee } \ast \eta_\ee) \big] = 0      @>  \ee \to 0^+ > \text{False in general}>   \partial_t u + \div \big[ u b(u) \big] = 0
\end{CD}
\end{equation}
To conclude, we make two remarks concerning i) the ``diagonal" convergence, which can be tracked explicitly in the case of regular initial data, and ii) some 
open questions. 
\begin{remark}
\label{p:diagonal}
Under the assumptions of Theorem~\ref{t:convbadata}, let $u_{\ee \nu}$ satisfy~\eqref{e:cpr}, let $u$ be the Kru{\v{z}}kov entropy admissible solution of~\eqref{e:clcpr}, and fix $p$ satisfying~\eqref{e:condizionip}. Combining Kru{\v{z}}kov's Theorem with Theorem~\ref{t:convbadata} and by a diagonal argument we infer that there is a sequence $(\ee_n, \nu_n)$ such that $\ee_n \to 0^+$, $\nu_n \to 0^+$ and $u_{\ee_n \nu_n} \to u$ strongly in 
$L^{\infty}_{\mathrm{loc}}([0, + \infty[; L^p(\R^d))$, as $n \to + \infty$.
In the case when the initial datum is sufficiently regular, namely $\overline u \in W^{1,p}(\R^d)$, we explicitly determine a coupling $\ee\leq e^{-C \nu^{-\beta}}$ (for constants $C>0$ and $\beta>0$ specified later) under which the above diagonal convergence holds true
(see Theorem~\ref{e:convreg} below).
\end{remark}
\begin{remark}
In the last few years, several authors have studied nonlocal traffic models like~\eqref{e:intro} in the case when $d=1$ and the convolution term $w\ast \eta$
only takes into account the downstream traffic density, see for instance~\cite{BlandinGoatin}. The convolution term in these models does not satisfy the regularity requirement in~\eqref{e:eta} because it is piecewise smooth with one or two discontinuity points. We are confident that  the regularity requirement in~\eqref{e:eta} can be weakened and that Theorem~\ref{t:convbadata} can be extended to the viscous version of the model described in~\cite{BlandinGoatin}. Note, however, that the counterexamples discussed in~\S~\ref{s:ce} do not apply to the model discussed in~\cite{BlandinGoatin}: whether or not the singular limit from~\eqref{e:nlcpr} to~\eqref{e:clcpr} can be rigorously justified in this case is presently an open problem. Partial results have been recently obtained in~\cite{ColomboCrippaSpinolo:blowup}. 
\end{remark}
\subsection*{Paper outline}
The paper is organized as follows. In~\S~\ref{s:preliminary} we establish well-posedness of the Cauchy problem~\eqref{e:cpr}, we  slightly extend known well-posedness results for~\eqref{e:nlcpr} and we establish Proposition~\ref{p:vnltonl}. In~\S~\ref{s:regdata} we establish Theorem~\ref{t:convbadata} under the additional assumption that the initial datum $\bar u$ is regular. In~\S~\ref{s:teorema11} we complete the proof of Theorem~\ref{t:convbadata} and in~\S~\ref{s:ce} we discuss the counterexamples to the nonlocal to local limit from~\eqref{e:nlcpr} to~\eqref{e:clcpr}. 
\subsection*{Notation}
For the readers' convenience, we recall here the main notation used in the present paper. 

We denote by $C(a_1, \dots, a_N)$ a constant only depending on the quantities $a_1, \dots, a_N$. Its precise value can vary from occurrence to occurrence. %We denote by $C$ a universal constant (i.e., a number), and again its precise value can vary from occurrence to occurrence. 

\subsubsection*{General mathematical symbols}
\begin{itemize}
\item $f \ast g$: the convolution of the functions $f$ and $g$, computed with respect to the variable $x$ only. 
\item $\div f$: the divergence of the vector field $f$, computed with respect to the $x$ variable only.
\item $\mathbf{1}_E$: the characteristic function of the measurable set $E$. 
\item $|E|$: the Lebesgue measure of the measurable set $E$. 
\item $L^p$: the Lebesgue space $L^p(\R^d)$, $p \in [1, + \infty]$. 
\item $\| \cdot \|_{L^p}$: the standard norm in $L^p (\R^d)$. 
\end{itemize}
\subsubsection*{Symbols introduced in the present paper}
\begin{itemize}
\item $b$: the vector-valued function satisfying~\eqref{e:v}.
\item $L$: the Lipschitz constant in~\eqref{e:v}. 
\item $\eta, \eta_\ee$: the convolution kernel in~\eqref{e:eta} and~\eqref{e:etaee}.
\item $u$: the entropy solution of the conservation law~\eqref{e:clcpr}.
\item $u_\ee$: the solution of the nonlocal nonviscous problem~\eqref{e:nlcpr}.
\item $u_\nu$: the solution of the local viscous problem~\eqref{e:vclr}.
\item $u_{\ee \nu}$: the solution of the nonlocal nonviscous problem~\eqref{e:cpr}.
\item $G, G_\nu$: the heat kernel in~\eqref{e:G} and~\eqref{e:Gnu}.
%\item $A$: the same constant as in~\eqref{e:ansatz}. 
%\item $K$: the constant in~\eqref{e:kappa2}. 
\item $S^{\ee \nu}_t, \ S^{\nu}_t:$ the semigroups defined in~\eqref{e:semigroups}. 
\end{itemize}
\begin{remark}
\label{r:b}
Consider the Lipschitz continuous function $b: \R \to \R^d$ in~\eqref{e:nlcpr},~\eqref{e:clcpr},~\eqref{e:cpr} and~\eqref{e:vclr}.  We can assume, with no loss of generality, that  $b(0)= 0$. Indeed, assume that this is not the case and that $b(0) =  \xi \neq 0$. Assume furthermore that the function $u_{\ee \nu}$ satisfies~\eqref{e:cpr}, then we can set  
$$
    \tilde u_{\ee \nu} (t, x) : = u_{\ee \nu} (t, x-  \xi t),
$$
and obtain that $ \tilde u_{\ee \nu}$ satisfies 
$$
    \partial_t  \tilde u_{\ee \nu} + \div \Big[  \tilde u_{\ee \nu}  \tilde b( \tilde u_{\ee \nu} \ast \eta_\ee)
     \Big] = \nu \Delta  \tilde u_{\ee \nu}
% $$
 %provided that 
% $$
\qquad \mbox{where}\qquad
    \tilde b( \tilde u_{\ee \nu} \ast \eta_\ee): =  b( \tilde u_{\ee \nu} \ast \eta_\ee)-  \xi. 
 $$
For this reason  in the following we assume that $b$ satisfies 
\begin{equation}
\label{e:v}
      b(0)= 0, \qquad 
     |b(x) - b(y)| \leq L |x -y| \; \text{for every $x, y \in \R$}. 
\end{equation}
\end{remark}
\begin{remark}
Theorem~\ref{t:convbadata} states that $u_{\ee \nu} \to u_\nu$ strongly in  $L^{\infty}_{\mathrm{loc}}([0, + \infty[; L^p)$, hence to establish the thesis it suffices to prove that, for every $T>0$, 
$u_{\ee \nu} \to u_\nu$ strongly in  $L^{\infty}([0, T]; L^p)$. 
A similar remark applies to Proposition~\ref{p:vnltonl} and to the other positive results, which are all local in time. 
To simplify the notation, in the following we take $T=1$. % The general case $T >0$ is completely analogous.  
\end{remark}
\section{Preliminary results: well-posedness of the viscous and nonviscous Cauchy problem with nonlocal fluxes}
\label{s:preliminary}
This section is organized as follows: in~\S~\ref{ss:wpviscous} for the sake of completeness we establish well-posedness of the nonlocal viscous Cauchy problem~\eqref{e:cpr}. We rely on fairly standard energy estimates and we apply a fixed point argument. In~\S~\ref{ss:exunice} we establish a uniqueness result for the nonlocal conservation law~\eqref{e:nlcpr} that slightly extends previous results in~\cite{ACT, ColomboHertyMercier,CLM,KePl}. Finally, in~\S~\ref{ss:proofp} we establish the proof of the nonlocal vanishing viscosity result stated in Proposition~\ref{p:vnltonl}.  Since in this case the nonlocal parameter $\ee>0$ is kept constant, weak convergence suffices to pass to the limit.  
\subsection{Well-posedness of the viscous Cauchy problem with a nonlocal flux}
\label{ss:wpviscous}
We establish the following well-posedness result. 
\begin{theorem}
\label{t:wp}        
        Let $\bar u \in L^1 \cap L^\infty(\R^d) $ and let $b$ satisfy~\eqref{e:v}.  Then the nonlocal viscous Cauchy problem~\eqref{e:cpr} has a
        distributional solution $u_{\ee \nu}$, unique in the class~\eqref{e:elleunoinfinito}-\eqref{e:spazioparabolico}, that satisfies 
        \begin{equation}
        \label{e:elleunoinfinito}
            \| u_{\ee \nu}(t, \cdot) \|_{L^1} \leq \| \bar u \|_{L^1}, 
            \quad 
            \| u_{\ee \nu}(t, \cdot) \|_{L^\infty} \leq C(\| \bar u \|_{L^\infty},
            \| \bar u \|_{L^1}, \| \nabla \eta \|_{L^\infty}, L,  d, \ee), 
            \quad \text{for every $t \in [0, 1]$,} 
        \end{equation}
        \begin{equation}
        \label{e:spazioparabolico}
    \partial_t u_{\ee \nu} \in L^2 ([0, 1]; H^{-1} (\R^d)), \quad 
    u_{\ee \nu} \in L^2 ([0, 1]; H^1 (\R^d)). 
        \end{equation} 
        \end{theorem}
\begin{remark}
\label{r}
The function $u_{\ee \nu}$ is in principle only defined for a.e. $(t, x)$. However, the regularity~\eqref{e:spazioparabolico} implies that, up to changing $u_{\ee \nu}$ in a set of measure $0$ in $[0,1]\times \R^d$, we can assume that $u_{\ee \nu} \in C^0 ([0, 1]; L^2(\R^d))$. In the following, we always consider this $L^2$-continuous representative; in this way the function $u_{\ee \nu}$ is well-defined \emph{for every $t$} and the estimates~\eqref{e:elleunoinfinito} hold for every $t$. 
\end{remark}

\begin{proof}[Proof of Theorem~\ref{t:wp}]
To simplify the notation, let $\ee=1$, $\nu =1$ and consider the Cauchy problem
\begin{equation}
\label{e:mdcp}
\left\{
\begin{array}{ll}
  \partial_t v + \div \big[ v b(v \ast \eta) \big] = \Delta v \\
  v (0, x) = \bar u(x). \\
\end{array}
\right.
\end{equation}
The proof straightforwardly extends to the general case and relies on a classical fixed point argument that we sketch below.\\
{\sc Step 1:} we introduce the functional setting. We fix a constant $0<\tau<1$, to be determined in the following, and we 
define the set $X$ by setting 
\begin{equation}
\label{e:X}
   X: = \big\{ z \in C^0 ([0, \tau] ; L^2 (\R^d)): \; \| z (t, \cdot) \|_{L^1} \leq \|\bar u \|_{L^1} \; \forall \, t \in [0, \tau]
   \big\} .
\end{equation}
We fix a function $\zeta \in X$ and we consider the Cauchy problem 
\begin{equation}
\label{e:auxcp}
\left\{
\begin{array}{ll}
  \partial_t z + \div \big[ z b(\zeta \ast \eta) \big] = \Delta z \\
  z (0, x) = \bar u(x) . \\
\end{array}
\right.
\end{equation}
Since $\zeta$ is now fixed, the equation at the first line of the above system is a standard linear parabolic equation with
smooth coefficients. By using classical methods for evolution equations (see for instance~\cite[\S~7]{Evans}) one can show that~\eqref{e:auxcp} has a unique solution satisfying 
$$
    \partial_t z \in L^2 ([0, \tau]; H^{-1} (\R^d)), \quad 
    z \in L^2 ([0, \tau]; H^1 (\R^d)),   
$$ 
which implies that (up to re-defining $z$ on a negligible set of times) $z \in C^0 ([0, \tau]; L^2 (\R^d))$. In the following, we always identify $z$ and its $L^2$-continuous representative, in such a way that $z(t, \cdot)$ is well-defined for every $t>0$. We define the map $T$ by setting $T(\zeta)=z$, where $z$ is the solution of~\eqref{e:auxcp}.   \\
{\sc Step 2:} we show that the map $T$ defined as in {\sc Step 1} attains values in $X$. We fix a regular function $\beta: \R \to \R$ and by multiplying the equation at the first line of~\eqref{e:auxcp} times $\beta' (z)$ we get 
\begin{equation}
\label{e:eqbeta}
    \partial_t \big[ \beta (z)\big] + \div \big[ b(\zeta \ast \eta) \beta (z)
    \big]+
    \div \big[b(\zeta \ast \eta) \big] \big( z \beta'(z) - \beta (z) \big)=
    \div \big[ \nabla z \beta' (z) \big] - \beta''(z) |\nabla z|^2.
\end{equation}
We point out that by \eqref{e:v} and \eqref{e:X}
$$
    | \div \big[b(\zeta \ast \eta) \big]|  
   {\leq}
    C(L, d)  \| \nabla \eta \|_{L^\infty} \| \bar u \|_{L^1}.
$$
By space-time integrating~\eqref{e:eqbeta} we get 
\begin{equation}
\begin{split}
\label{e:eqapriori}
       \int_{\R^d} \! \! \beta(z) (t, \cdot) dx & -
        \int_{\R^d} \! \! \beta(\bar u) dx + \int_0^t \! \! \int_{\R^d} 
       \beta''(z) |\nabla z|^2 dx ds \\
       & \leq 
                  C(L, d) \| \bar u \|_{L^1}  
          \| \nabla \eta \|_{L^\infty}  
        \int_0^t \! \! \int_{\R^d}  
        |  z \beta'(z) - \beta (z) |
        dx ds, \qquad \text{for every $t \in [0, \tau]$.} \\ 
        \end{split}
\end{equation}
By applying~\eqref{e:eqapriori} with $\beta(z)=z^2$ and using the Gr\"onwall Lemma we get that for every $t \in [0, \tau]$
\begin{equation}
\label{e:puntofissoelledue}
       \| z(t, \cdot) \|_{L^2} \leq C(L, d, \| \bar u \|_{L^1},  
          \| \nabla \eta \|_{L^\infty}  ) \| \bar u \|_{L^2}. 
\end{equation}
Also, by using~\eqref{e:eqapriori} and choosing a suitable approximation of $ \beta(z) = |z|$, we get 
\begin{equation}
\label{e:elleuno}
        \int_{\R^d} \! \! |z| (t, \cdot) dx -
        \int_{\R^d} \! \! |\bar u| dx \leq 0.  
\end{equation}
This implies that the solution of~\eqref{e:auxcp}, i.e. $T(\zeta)$, belongs to the set $X$ defined as in~\eqref{e:X}. \\
{\sc Step 3:} we show that the map $T$ defined as in {\sc Step 1} is a contraction provided that $\tau$ is sufficiently small. 
We fix $\zeta_1, \zeta_2 \in X$ and we term $z_1=T(\zeta_1)$ and $z_2= T(z_2)$. First, we point out that owing to the Young Inequality  
\begin{equation} \label{e:stimab1b2}
 \| b(\zeta_1 \ast \eta) -  
       b(\zeta_2 \ast \eta)  \|_{L^\infty} 
    \leq
      L \| (\zeta_1- \zeta_2)\ast \eta \|_{L^\infty} \leq L
       \| (\zeta_1- \zeta_2) \|_{C^0([0, \tau]; L^2 )} 
        \|  \eta \|_{L^2},
\end{equation}
for every $t\in [0, \tau]$. 
By subtracting the equation for $z_2$ from the equation for $z_1$ we get 
$$
   \partial_t \big[ z_1 - z_2 \big]
   + \div \Big[ 
   [z_1 -z_2]  b(\zeta_1 \ast \eta) + z_2 
   \big[ b(\zeta_1 \ast \eta) -
   b(\zeta_2 \ast \eta)  \big] \Big] = \Delta [z_1 - z_2 ]. 
$$  
By arguing as in {\sc Step 2} and recalling that $z_2 \equiv z_1$ at $t=0$, we arrive at 
\begin{equation}
\label{e:contra}
\begin{split}
     \int_{\R^d} & \! \! |z_1 - z_2|^2 (t, \cdot) dx 
      \leq 
      C(L, d, \| \bar u \|_{L^1},  
          \| \nabla \eta \|_{L^\infty} )  
    \int_0^t  \int_{\R^d} \! \! |z_1 - z_2|^2 (s, \cdot) dx  ds \\
    & + 2
    \left| \int_0^t \int_{\R^d} \div \Big[ z_2 
   \big[ b(\zeta_1 \ast \eta) -
   b(\zeta_2 \ast \eta)  \big] \Big] (z_1 -z_2) dx ds  \right| 
   - 2    \int_0^t \int_{\R^d}| \nabla [z_1 - z_2] |^2 dx ds .
   \end{split}
 \end{equation}
Next, we point out that 
  \begin{equation*}
\begin{split}
       \left| \int_0^t  \int_{\R^d} \right. & \left. \phantom{\int} \!\!\!\!\!\!\! \div \Big[ z_2 
   \big[ b(\zeta_1 \ast \eta) -
   b(\zeta_2 \ast \eta)  \big] \Big] (z_1 -z_2) dx ds \right| =
       \left| \int_0^t \int_{\R^d}   z_2 
   \big[ b(\zeta_1 \ast \eta) -
   b(\zeta_2 \ast \eta)  \big] \cdot  \nabla[z_1 -z_2] dx ds  \right| \\
   & \leq  \frac{1}{2}
    \int_0^t \int_{\R^d} z_2^2 
    \big| b(\zeta_1 \ast \eta) -
   b(\zeta_2 \ast \eta)  \big|^2 dx ds + \frac{1}{2} 
   \int_0^t \int_{\R^d} | \nabla[z_1 -z_2]|^2 dx ds.
\end{split}
\end{equation*}
To control the first term in the right hand side of the above expression we combine~\eqref{e:puntofissoelledue} and~\eqref{e:stimab1b2}.  
By plugging the above inequality into~\eqref{e:contra} we then arrive at 
 \begin{equation*}
\begin{split}
     \int_{\R^d} & \! \! |z_1 - z_2|^2 (t, \cdot) dx 
      \leq 
      C(L, d, \| \bar u \|_{L^1}, \eta ) \left[  
    \int_0^t  \int_{\R^d} \! \! |z_1 - z_2|^2 (s, \cdot) dx  ds  + \tau 
     \| \bar u \|^2_{L^2}    
     \| (\zeta_1- \zeta_2) \|^2_{C^0([0, \tau]; L^2 )} \right]
\end{split}   
\end{equation*}   
and owing to the Gr\"onwall Lemma and recalling that $z_1 = T (\zeta_1)$, $z_2 = T (\zeta_2)$ this implies that $T$ is a contraction provided that $\tau$ is sufficiently small. To establish existence and uniqueness on the interval $[0, 1]$ we iterate the above argument a finite number of times. \\
{\sc Step 4:} we establish the $L^\infty$ estimate. We recall~\eqref{e:auxcp}, we set 
$$
    \Xi: = \| \mathrm{div} [ b(\zeta \ast \eta)  ] \|_{L^\infty}
$$
and we point out that the solution $z$ of the Cauchy problem~\eqref{e:auxcp} satisfies 
\begin{equation}
\label{e:linftyaux2}
    \| z (t, \cdot) \|_{L^\infty}
    \leq \| \bar u \|_{L^\infty} \exp (\Xi t), 
    \quad \text{for every $t$}. 
\end{equation}
The proof of the above estimate is  standard, and can be found for instance in~\cite[Lemma 3.4]{ProcHyp2012}.
By construction, the solution of~\eqref{e:mdcp} satisfies~\eqref{e:auxcp} provided that $\zeta= z$. If this is the case, by~\eqref{e:elleuno}
$$
     \Xi = \| \mathrm{div} [ b(z \ast \eta)  ] \|_{L^\infty} \leq 
     C(L, d) \| z \|_{L^1} \| \nabla \eta \|_{L^\infty} 
     \leq
      C(L, d) \| \bar u \|_{L^1} \| \nabla \eta \|_{L^\infty} 
$$
and owing to~\eqref{e:linftyaux2} this establishes the $L^\infty$ estimate in~\eqref{e:elleunoinfinito}.  
\end{proof}
\subsection{Well-posedness of the Cauchy problem for a continuity equation with nonlocal flux }
\label{ss:exunice}
In this section we establish an existence and uniqueness result that slightly 
extends the well-posedness result in~\cite{ACT} (see also~\cite{ColomboHertyMercier,KePl}).   
\begin{proposition}
\label{p:exunice}
Assume that $b$ and $\eta$ satisfy~\eqref{e:v} and~\eqref{e:eta}, respectively, and that $\bar u \in L^1 \cap L^\infty $. 
Then the Cauchy problem~\eqref{e:nlcpr} has a distributional solution that satisfies 
$$
u_\ee \in L^\infty_{\mathrm{loc}} ([0, + \infty[; L^\infty) \cap C^0 ([0, + \infty[; L^1).
$$  
Also, the solution is unique in the class of locally bounded, distributional solutions. 
\end{proposition}
In the following we identify $u_\ee$ and its $L^1$ strongly continuous representative. 
The existence part of Propostion~\ref{p:exunice} is a consequence of the analysis in~\cite[\S~2]{ACT} (see also~\cite{ColomboHertyMercier,CLM}). The relatively new part is the uniqueness: indeed, in~\cite{ACT,KePl} uniqueness is established in a slightly more restrictive class, while in \cite{CLM} it is established only for nonnegative data. More precisely, in~\cite{ColomboHertyMercier} it is shown that there is a unique solution $u$, in the sense of~Kru{\v{z}}kov~\cite{Kruzkov}, of the conservation law
\begin{equation}
\label{e:kruzkov3}
\left\{
\begin{array}{ll}
  \partial_t u_{\ee } + \div \big[ u_{\ee } g_\ee \big] = 0 \\
  u_{\ee } (0, x) = \bar u(x) \\
\end{array}
\right.
\end{equation} 
provided that the function  $g_\ee$ is given by $g_\ee : = b(u_\ee \ast \eta_\ee)$ (see~\cite[Definition 2.1]{ACT}) and that the initial datum is quite regular, namely $\bar u $ has bounded total variation.
On the other hand, Proposition~\ref{p:exunice} states the uniqueness of locally bounded distributional solutions. 
\begin{proof}[Proof of Proposition~\ref{p:exunice}, uniqueness]
Let $u_\ee$ be a distributional solution of~\eqref{e:nlcpr}. Then $u_\ee$ is a distributional solution of~\eqref{e:kruzkov3}. Next, we observe that the first line of~\eqref{e:kruzkov3} is a continuity equation with a regular in space coefficient $g_\ee$. Every locally bounded distributional solution of~\eqref{e:kruzkov3} is therefore renormalized, meaning that for every $\beta \in C^1 (\R)$ we have that $\beta(u)$
is a distributional solution of 
\begin{equation}
\label{e:renormalized}
   \partial_t \big[\beta(u)  \big] +  \div \big[ \beta (u) g_\ee \big]+
   \div g_\ee \big[ \beta'(u) u - \beta (u)  \big]=0. 
\end{equation}
This is for instance an application (in a very easy case) of the DiPerna-Lions-Ambrosio theory, and we refer to~\cite{Ambrosio,DiPernaLions} for that. 
Equation~\eqref{e:renormalized} implies that, up to redefining $u_\ee$ in a negligible set, $u_\ee$ belongs to $C^0 ([0, + \infty[;L^1)$ and it is a Kru{\v{z}}kov solution of the conservation law~\eqref{e:kruzkov3}: this can be proved by arguing as in the proof of Corollary 3.14 in~\cite{DL:note}. 
Since by~\cite[Theorem 3.2]{KePl}
distributional solutions of~\eqref{e:nlcpr} in $C^0 ([0, + \infty[;L^1)$ are unique, this concludes the proof of 
Proposition~\ref{p:exunice}. 
\end{proof} 
\subsection{Proof of Proposition~\ref{p:vnltonl}}
\label{ss:proofp}
	Let $\ee >0$. We consider $u_{\ee \nu}$ satisfying~\eqref{e:cpr} and we recall the $L^\infty$ estimate in~\eqref{e:elleunoinfinito}. We fix a sequence $\nu_n$ and a function $u_\ee \in L^\infty ([0, 1] \times \R^d)$ such that 
	\begin{equation}
	\label{e:weaks}
	u_{\ee \nu_n} \weaks u_\ee \; \text{weakly$^\ast$ in $L^\infty([0, 1] \times \R^d)$ as $\nu_n \to 0^+$.}
	\end{equation}
	%Assume for a moment that we have shown that $u_\ee$ is a distributional solution of~\eqref{e:nlcpr}. 
	
 	We claim that $u_\ee$ is a distributional solution of~\eqref{e:nlcpr}. To take the limit in the distributional formulation of \eqref{e:cpr} and prove this claim, it is enough to show that 
	\begin{equation}
	\label{e:regconv}
	u_{\ee \nu_n} \ast \eta_\ee \to u_\ee \ast \eta_\ee \; \text{strongly in $L^1_{\loc}([0, 1] \times \R^d)$.}
	\end{equation}
	If the claim is true, since bounded, distributional solutions of~\eqref{e:nlcpr} are unique by  Proposition~\ref{p:exunice}, the whole family $u_{\ee \nu_n}$ converges to $u_\ee$ weakly$^\ast$ in $L^\infty([0, 1] \times \R^d)$, proving Proposition~\ref{p:vnltonl}. 
	
	To show \eqref{e:regconv} we point out first that by \eqref{e:elleunoinfinito} for every $t\in [0,1]$
	\begin{equation}
	\label{e:c01}
	\| [u_{\ee \nu} \ast \eta_\ee] (t, \cdot) \|_{L^\infty}  +	\| \nabla [u_{\ee \nu} \ast \eta_\ee ]  (t, \cdot) \|_{L^\infty} 
	\leq   \| u_{\ee \nu} (t, \cdot) \|_{L^\infty}  \| \eta_\ee  \|_{W^{1,1}}{\leq} 
	C(\| \bar u \|_{L^\infty},
	\| \bar u \|_{L^1},  \eta, L, d, \ee).
	\end{equation}
The time derivative of $u_{\ee \nu} \ast \eta_\ee$ is obtained by convolving every term in~\eqref{e:cpr} with $\eta_\ee$, that is 
	$$
	\partial_t [u_{\ee \nu} \ast \eta_\ee] (t, x) = - \div \big[\eta_\ee\ast (u_{\ee \nu} b (u_{\ee \nu})) \big]+ \nu  u_{\ee \nu} \ast \Delta \eta_\eps.	$$
%	To show this rigorously, observe that for any $\psi \in C^\infty_c (\R)$ the distributional formulation of \eqref{e:cpr} gives
%	\begin{equation*}
%	\begin{split}
%	\int_0^1 & \psi' [u_{\ee \nu} \ast \eta_\ee]  (t, x) dt =
%	\int_0^1  \psi' \int_{\R^d} 
%	u_{\ee \nu} (t, y)  \eta_\ee (t, x-y) dy dt \\ & 
%	\stackrel{\eqref{e:cpr}}{=} 
%	\int_0^1   \int_{\R^d} \psi
%	u_{\ee \nu} b (u_{\ee \nu} \ast \eta_\ee)  \nabla \eta_\ee (t, x-\cdot) dy dt - \nu  \int_0^1   \int_{\R^d} \psi \
%	u_{\ee \nu}  \Delta \eta_\ee ( x-\cdot) dy dt.
%	\end{split}
%	\end{equation*}
	By using~\eqref{e:v}, \eqref{e:etaee} and~\eqref{e:elleunoinfinito} we conclude that 
	\begin{equation}
	\label{e:c03}
	\| \partial_t [u_{\ee \nu} \ast \eta_\ee] (t, \cdot) \|_{L^\infty } \leq 
	\|\nabla\eta_\ee \|_{L^\infty } \|u_{\ee \nu} b (u_{\ee \nu})\|_{L^1 }+ \nu \| u_{\ee \nu} \|_{L^1 }\| \Delta \eta_\eps \|_{L^\infty } \leq
	C(  \| \bar u \|_{L^\infty}, 
	\| \bar u \|_{L^1}, \eta, L, d, \ee). 
	\end{equation}
	Finally, we combine~\eqref{e:c01} and~\eqref{e:c03} and we apply the Ascoli-Arzel\`a Theorem: there is a continuous function $w$ such that, up to subsequences (that we do not re-label)  
	$u_{\ee \nu_n} \ast \eta_\ee \to w$ uniformly on compact sets 
		of $[0,1] \times \R^d$.
For any $\phi \in C^\infty_c ([0, 1] \times \R^d)$, terming $\check \eta_\ee (z):= \eta_\ee (-z)$ and by \eqref{e:weaks}, we have
	\begin{equation*}
	\begin{split}
	\int_0^1 \int_{\R^d}  \phi (t, x) w (t, x) dx dt  
	&= \lim_{n\to \infty}\int_0^1 \int_{\R^d} \phi [u_{\ee \nu_n} \ast \eta_\ee ] dx dt 
	%  =\int_0^1 \int_{\R^d}  \int_{\R^d} \phi (t, x) u_{\ee \nu_n} (t, y)  \eta_\ee (x- y)  dy dx dt
   = \lim_{n\to \infty} \int_0^1 \int_{\R^d} u_{\ee \nu_n} 
	[
	\phi \ast \check \eta_\ee ] dy dt 
	\\&= \int_0^1 \int_{\R^d} u_{\ee}
	[
	\phi \ast \check \eta_\ee ] dy dt
	=  \int_0^1 \int_{\R^d}  \phi [u_{\ee } \ast \eta_\ee ] dx dt.  
	\end{split}
	\end{equation*}
By the arbitrariness of $\phi$ we deduce that 	$w= u_\ee \ast \eta_\ee$ a.e. in $[0, 1] \times \R^d$  and  
	hence we prove \eqref{e:regconv}. 
\qed
\section{Convergence of the nonlocal viscous approximation for regular data}
\label{s:regdata}
In this section we establish the nonlocal to local limit asserted in Theorem~\ref{t:convbadata} assuming more restrictive conditions on the initial data. This intermediate result is pivotal to the proof of Theorem~\ref{t:convbadata}.
\begin{theorem}
\label{t:convreg} Fix $0<\nu<1/4$. Assume that $b$ satisfies~\eqref{e:v} and $\eta_\ee$ satisfies~\eqref{e:eta} and~\eqref{e:etaee}. Let $p$ satisfy~\eqref{e:condizionip}, let $\beta=(p+d)/(p-d)$ and 
assume that $\bar u \in L^1 \cap L^\infty \cap W^{1, p} (\R^d)$. Let $u_{\ee \nu}$ and $u_\nu$  be the solutions of the Cauchy problems~\eqref{e:cpr} and~\eqref{e:vclr}, respectively. Then there exists $C:=C(d, p, L, \| \bar u \|_{L^\infty}, \| \bar u \|_{W^{1,p}})$ such that, 
if $\ee\leq e^{-C \nu^{-\beta}}$, we have
\begin{equation}
\label{e:convreg}
        \| u_{\ee \nu} (t, \cdot) - u_\nu(t, \cdot)  \|_{L^p }
        \leq  \ee e^{C \nu^{-\beta}},
        \quad \text{for every $t \in [0, 1]$}.
\end{equation}  
This, in particular, implies that $u_{\ee \nu} \to u_\nu$ strongly in $L^{\infty}_{\mathrm{loc}}([0, + \infty[; L^p(\R^d))$ as 
$\ee \to 0^+$. 
\end{theorem}
%{\color{red} tracciare dipendenza esplicita da $\nu$}
To establish Theorem~\ref{t:convreg} we introduce the function 
\begin{equation}
\label{e:zetaee}
  z_\ee : = u_{\ee \nu }- u_\nu. 
\end{equation}
Note that, to simplify the notation, we do not explicitly indicate the dependence of $z_\ee$ on $\nu$. Next, we compute the equation satisfied by $z_\ee$ and we perform careful a-priori estimates by extensively using the Duhamel representation formula. 

The proof of Theorem~\ref{t:convreg} is organized as follows: in~\S~\ref{ss:prel} we review some basic results concerning viscous conservation laws. In~\S~\ref{ss:proofteo} we provide the proof of Theorem~\ref{t:convreg} by establishing precise a-priori 
estimates on the growth rate of the function $z_\ee$ in~\eqref{e:zetaee}.

\subsection{Preliminary results}
\label{ss:prel}
In~\S~\ref{sss:heat} we recall some basic results about heat kernels, in~\S~\ref{sss:viscous} we go over some a-priori estimates on solutions of viscous conservation laws that we need in the following. 
\subsubsection{Heat kernels}
\label{sss:heat}
We recall some basic properties of the heat kernel $G: ]0, + \infty[ \times \R^d \to \R$
\begin{equation}
\label{e:G}
    G(t, x) : = C(d) \frac{1}{ t^{d/2}} \exp \left( - \frac{|x|^2}{ 4 t}\right).
\end{equation}
The normalization constant $C(d)$ is chosen in such a way that 
$ \| G (t, \cdot) \|_{L^1 (\R^d)} =1,$ for every $t>0$. Since 
$$
    | \nabla G (t, x) | = C(d) \frac{|x|}{t^{d/2+1}}  \exp \left( - \frac{|x|^2}{ 4 t}\right),
$$
by using spherical coordinates and by making the change of variables $\rho'= \rho /2 t^{1/2} $ we get 
\begin{equation}
\label{e:nablaG}
\begin{split}
    \| \nabla G (t, \cdot) \|_{L^{q} (\R^d)} & = 
   C (d, q) \left( \int_0^\infty 
   \rho^{d-1} \frac{\rho^{q}}{t^{q (d/2+1)}} e^{-\frac{q \rho^2}{4t}} d\rho 
   \right)^{1/q} 
 %   = \Big[ z = \rho /2 t^{1/2} \Big] \\ &=C(d, q)
 %  \left( \int_0^\infty z^{d-1+q} e^{- q z^2}
%  t^{\frac{d- q (d+1)}{2}}  d \rho
 %  \right)^{1/q}
  =
   C(d, q) t^{\frac{d- q (d+1)}{2q}} = C(d,q) t^\alpha. 
\end{split}
\end{equation}
For later use we have set 
\begin{equation}
\label{e:alpha} \alpha: = \frac{d- q (d+1)}{2q} %= -(d+p)/(2p)\in( -1 , 0)
\end{equation}
in the formula above. Given $\nu>0$, we introduce the kernel 
\begin{equation}
\label{e:Gnu}
           G_\nu (t, x) : =G \left( {\nu}{t}, {x} \right)= \frac{1}{\nu^d} G \left( \frac{t}{\nu}, \frac{x}{\nu} \right),
\end{equation}
which is the fundamental solution of the equation $\partial_t u = \nu \Delta u$
and satisfies
\begin{equation}
\label{e:G1nu}
     \| G_\nu (t, \cdot) \|_{L^1} = 1, \qquad
%\end{equation}
%By direct computations, one gets 
%\begin{equation}
%\label{e:nablaGnu}
    \| \nabla G_\nu (t, \cdot) \|_{L^q} = 
    \| \nabla G (\nu t, \cdot) \|_{L^q} \stackrel{}{=}
     C(d, q)(\nu t)^\alpha. %^{\frac{d- q (d+1)}{2q}}.
\end{equation}

\subsubsection{A priori estimates on solutions of a viscous conservation law}
\label{sss:viscous}
The following lemma collects some classical a-priori estimates we need in the following. 
\begin{lemma}
\label{l:vcl}
Let $\nu \in (0,1)$. Assume $b$ satisfies~\eqref{e:v}, $\bar u \in L^1 (\R^d) \cap L^\infty (\R^d)$. The  solution of the Cauchy problem~\eqref{e:vclr} satisfies:
\begin{equation}
\label{e:controlp}
    \| u_\nu (t, \cdot) \|_{L^p }
    \leq \| \bar u \|_{L^p }, 
   \quad \text{for every  $t \in [0, 1]$ and every $p \in [1, + \infty]$.}
\end{equation}
Let $u_\nu$ and $w_\nu$ be the two solutions corresponding to the data $\bar u$ and $\bar w$, respectively.

Then we have the following stability estimate: for every $p \in [1, + \infty]$,
\begin{equation}
\label{e:stability}
      \| u_\nu (t, \cdot) - w_\nu (t, \cdot) \|_{L^p }
      \leq e^{C(d,p, L,  \| \bar u \|_{L^\infty}, 
      	\| \bar w \|_{L^\infty} )\nu^{-1}} 
      \| \bar u - \bar w \|_{L^p},
      \quad \text{for every $t \in [0, 1]$}.
\end{equation}
If we also require $\bar u \in  W^{1, p} (\R^d)$, then we have 
\begin{equation}
\label{e:ellepdx}
      \| \nabla u_\nu (t, \cdot) \|_{L^p} \leq e^{
      C(d,p, L, \|\bar u \|_{L^\infty})\nu^{-1}} \
      \| \nabla \bar u \|_{L^p}, 
      \quad \text{for every $t \in [0, 1]$}.
\end{equation} 
\end{lemma}
\begin{remark}
\label{r2}
Note that~\cite[Lemma~6.3.3]{Dafermos:book} implies that the function $u_\nu$ (which a priori is only defined for a.e.~$(t, x)$) has a representative such that the function $t \mapsto u(t, \cdot)$ is continuous from $[0, 1]$ to $L^1$ endowed with the strong topology. Here and in the following, we always identify $u_\nu$ and its $L^1$-continuous representative. 
\end{remark}    
\begin{proof}[Proof of Lemma~\ref{l:vcl}]
When $p=\infty$, the estimate \eqref{e:controlp} is a maximum principle, which is a classical result~\cite[\S~VI]{Dafermos:book}. The result for $p\in [1, + \infty[$ is also classical, but for the sake of completeness we provide a sketch of the proof. We rewrite the equation at the first line of~\eqref{e:vclr} in the quasi-linear form 
\begin{equation*}
    \partial_t u_\nu + f'(u_\nu) \div u_\nu = \nu \Delta u_\nu,
\qquad \mbox{where} \qquad f(u) : = u b(u). 
\end{equation*}
We set $\beta (u) : = |u|^p$ and we multiply the above equation by $\beta' (u_\nu)$. We arrive at 
$$
     \partial_t \big[ \beta(u_\nu) \big] + f'(u_\nu) \beta' (u_\nu) \div u_\nu =\nu \Delta 
     \big[\beta( u_\nu) \big] - \nu \beta''(u_\nu) |\nabla u_\nu|^2. 
$$ 
We fix a function $h: \R \to \R^d$ satisfying $h' = f' \beta'$ and we rewrite the above equation as 
$$
     \partial_t \big[ \beta(u_\nu) \big] + \div  \big[ h(u_\nu) \big] = \nu\Delta 
     \big[\beta( u_\nu) \big] - \nu \beta''(u_\nu) |\nabla u_\nu|^2. 
$$
Next, we integrate with respect to $x$ and use the convexity of the function $\beta$: we get 
$$
    \frac{d}{dt} \int_{\R^d} \beta (u_\nu) dx \leq 0, 
$$ 
which implies~\eqref{e:controlp}.
To prove~\eqref{e:stability}, we take the difference between the equation \eqref{e:vclr} for $u_\nu$ and $w_\nu$
$$    \partial_t (u_\nu- w_\nu) + \div \big((u_\nu-w_\nu)b(u_\nu)+w_\nu (b(u_\nu) - b(w_\nu))\big) = \nu \Delta (u_\nu- w_\nu).
$$
Multiplying by $p(u_\nu- w_\nu)|u_\nu- w_\nu|^{p-2}$ the previous equation and integrating in space we have
$$ \partial_t \|u_\nu- w_\nu\|_{L^p}^p = \int_{\R^d}p(u_\nu- w_\nu)|u_\nu- w_\nu|^{p-2} \Big[-\div \big((u_\nu-w_\nu)b(u_\nu)+w_\nu (b(u_\nu) - b(w_\nu))\big) + \nu \Delta (u_\nu- w_\nu)\Big] .
$$
Integrating by parts, using assumptions \eqref{e:v} on $b$ and \eqref{e:controlp} with $p=\infty$ we get
\begin{equation*}
\begin{split}
\partial_t \|u_\nu- w_\nu\|_{L^p}^p & \leq C(p) \int_{\R^d} |\nabla (u_\nu- w_\nu)| |u_\nu-w_\nu|^{p-1} |b(u_\nu)|+|u_\nu-w_\nu|^{p-2} |\nabla (u_\nu- w_\nu)| |w_\nu| |b(u_\nu) - b(w_\nu)| 
\\& \qquad \quad \qquad- \nu \int_{\R^d}  |\nabla (u_\nu- w_\nu)|^2 |u_\nu- w_\nu|^{p-2}
\\& \leq C_0 \int_{\R^d} |\nabla (u_\nu- w_\nu)| |u_\nu-w_\nu|^{p-1} - \nu \int_{\R^d}  |\nabla (u_\nu- w_\nu)|^2 |u_\nu- w_\nu|^{p-2},
\end{split}
\end{equation*}
for a suitable constant $C_0$. By Young inequality we have
$$\int_{\R^d} |\nabla (u_\nu- w_\nu)| |u_\nu-w_\nu|^{p-1}
\leq \frac{\nu}{2C_0} \int_{\R^d} |\nabla (u_\nu- w_\nu)|^2 |u_\nu- w_\nu|^{p-2}+ \frac{C_0}{2\nu} \int_{\R^d} |u_\nu-w_\nu|^{p},
$$
which implies 
$$ \partial_t \|u_\nu- w_\nu\|_{L^p}^p \leq {C}{\nu^{-1}} \|u_\nu- w_\nu\|_{L^p}^p. $$
The Gr\"onwall lemma allows to conclude the validity of \eqref{e:stability}.

By the characterization of Sobolev functions in terms of finite differences (notice that for $p=1$ it would involve functions of bounded variation, but we know a priori that $\nabla u_\nu \in L^{2}$ for every $t\in[0,1]$), we have
$$  \| \nabla u_\nu (t, \cdot) \|_{L^p} \leq C \sup_{h\in \R^d \setminus \{0\}}\frac{1}{|h|} \| u_\nu (t, \cdot) - u_\nu (t, \cdot+h) \|_{L^p}, 
\quad \text{for every $t \in [0, 1]$.}
$$ 
Applying the stability \eqref{e:stability} to $\bar u$ and $ \bar u(\cdot +h)$ we estimate the right-hand side
$$ \| \nabla u_\nu (t, \cdot) \|_{L^p} \leq e^{
	C(d, L, \|\bar u \|_{L^\infty})\nu^{-1}}
 \sup_{h\in \R^d \setminus \{0\}}\frac{1}{|h|} \| \bar u(\cdot) - \bar u (\cdot+h) \|_{L^p} 
\leq e^{
	C(d, L, \|\bar u \|_{L^\infty})\nu^{-1}} \| \nabla \bar u \|_{L^p}$$ 
for every $t \in [0, 1]$. This proves \eqref{e:ellepdx}.
\end{proof}
\begin{lemma}
\label{l:wconvolution}
Assume that $b$ satisfies~\eqref{e:v} and that $u_\nu$ 
satisfies~\eqref{e:ellepdx}. Assume furthermore that the convolution kernel $\eta_\ee$ satisfies~\eqref{e:eta} and~\eqref{e:etaee}. Then for every $p \in [1, + \infty[$
we  have
\begin{equation}
\label{e:elleduee}
  \| u_\nu(t, \cdot)  - \eta_\ee \ast u_\nu (t, \cdot) \|_{L^p}
  \leq  
  \ee  e^{C(d,p, L, \|\bar u \|_{L^\infty})\nu^{-1}}
         \| \nabla \bar u \|_{L^p}, 
         \quad \text{for every $t \in [0, 1]$}.  
\end{equation}
\end{lemma}
\begin{proof}
	By Jensen's inequality applied with respect to the probability measure $\eta_\ee\, dx$ and by the finite differences characterization of Sobolev functions, we get 
		\begin{equation}
		\label{e:elledueallaseconda}
		\begin{split}
		\| u_\nu(t, \cdot)  - \eta_\ee \ast u_\nu (t, \cdot) \|^p_{L^p} & 
		\stackrel{\eqref{e:eta},\eqref{e:etaee}}{=}
		\int_{\R^d} \left| 
		\int_{\R^d}
		\big[ u_\nu(t, x-y) - u_\nu(t, x) \big] \eta_\ee (y) dy \right|^p dx \\
			& \;\;\;\;\; \leq
	\int_{\R^d}  \int_{\R^d}  
	   \big| u_\nu(t, x-y) - u_\nu(t, x) \big|^p \eta_\ee (y)
		dy\, dx \\
		& \;\;\;\;\; \leq
		 \| \nabla u_\nu (t, \cdot) \|^p_{L^p(\R^d) }
		\int_{\R^d} \eta_\ee (y)
		| y|^p dy \\
		& \quad\stackrel{\eqref{e:ellepdx}}{\leq}
		\ee^{p} e^{C(d,p, L, \|\bar u \|_{L^\infty})\nu^{-1}}
		\| \nabla \bar u \|^p_{L^p(\R^d) } 
		\int_{\R^d}  \eta_\ee (y)  \Big(\frac{|y|}{\ee}\Big)^p dy. 
		\end{split}        
		\end{equation}
Since $\eta_\ee$ is supported where $|y|\leq \ee$, the last integrand in the right-hand side is estimated by $\|\eta_\ee\|_{L^1}=1$, which concludes the proof of~\eqref{e:elleduee}.  
	\end{proof}

\subsection{Proof of Theorem~\ref{t:convreg}}
\label{ss:proofteo}
%We can now give the proof of Theorem~\ref{t:convreg}. 
First, we recall that $u_{\ee \nu}$ is the solution of~\eqref{e:cpr} and $u_\nu$ is the solution of~\eqref{e:vclr} and we define $z_\ee$ as in~\eqref{e:zetaee}. 
Note that $z_\ee$ satisfies the equation 
\begin{equation}
\label{e:eqzeta}
  \partial_t z_\ee + \div \mathcal T_\ee  = \nu \Delta z_\ee
\end{equation}
where, thanks to \eqref{e:zetaee}, the term $\mathcal T_\ee$ is given by 
\begin{equation}
\label{e:cosaet}
\begin{split}
  \mathcal T_\ee :=& \; u_{\ee \nu } b(u_{\ee \nu } \ast \eta_\ee) - u_\nu b(u_\nu) =
  u_{\ee \nu } \big[  b(u_{\ee \nu } \ast \eta_\ee) - b(u_\nu) \big] + [u_{\ee \nu }-  u_\nu] b(u_\nu) \\
  =& \;
  [z_\ee + u_\nu] \big[  b\big([z_\ee + u_\nu]  \ast \eta_\ee \big) - b(u_\nu) \big] + z_\ee b(u_\nu). 
\end{split}
\end{equation}
We now proceed as follows: in~\S~\ref{sss:apriori} we establish some a-priori estimates on $z_\ee$, which are the key point in the proof, and in~\S~\ref{sss:concludo} we conclude the proof of Theorem~\ref{t:convreg}. 
\subsubsection{A-priori estimates on $z_\ee$}
\label{sss:apriori}
We establish a-priori estimates on the solution of the Cauchy problem
\begin{equation}
\label{e:cpzeta}
\left\{
\begin{array}{ll}
         \partial_t z_\ee + \div \mathcal T_\ee  = \nu \Delta z_\ee \\
         z_\ee (0, x) = z_0 (x). \\
\end{array}
\right.
\end{equation}

\begin{lemma}
\label{l:zeta}
Let $b$ satisfy~\eqref{e:v}, $0 < \ee, \nu \leq{1}/{4}$,  $\eta_\ee$ as in~\eqref{e:eta} and~\eqref{e:etaee},  $z_0 \in L^p (\R^d)$ with $p$ as in~\eqref{e:condizionip}, and $\beta = (p+d)/(p-d)$. Assume furthermore that  the function $u_\nu$ satisfies~\eqref{e:controlp} and~\eqref{e:elleduee}.
Then there exist ${c_0:= c_0 (d, p, L, \| \bar u \|_{L^\infty})>0}$ and $\tau_0:=\tau_0(d, p, L, \| \bar u \|_{L^\infty}, \| \bar u \|_{W^{1,p}})>0$, such that if
\begin{equation}\label{hp:eps-z0}
		\| z_0 \|_{L^p} \leq {1}/{4}, \qquad \ee \leq e^{-c_0\nu^{-1}}/4
\end{equation}
 the solution of the Cauchy problem~\eqref{e:cpzeta} starting from $z_0$
satisfies 
\begin{equation}
\label{e:ansatz}
    \| z_\ee (t, \cdot) \|_{L^p } \leq  2 [\| z_0 \|_{L^p} + e^{c_0\nu^{-1}}\ee],  \quad \text{for every $t \in [0, \tau_0 \nu^\beta]$.}
\end{equation}
\end{lemma}
\begin{proof}
%First, we point out that, if $A>1$, then~\eqref{e:ansatz} holds provided that $t$ is sufficiently small. We now show that we can choose $A$ and $\tau$ in such a way that they do not depend on $\ee$. 
Let $c_0$ be the maximum between the constant in Lemma~\ref{l:wconvolution} and $1$. Set 
\begin{equation}
\label{e:deftau}
    \tau : = \sup \big\{ t \in [0, 1]: 
    \| z_\ee (s, \cdot) \|_{L^p } \leq 2 [\| z_0 \|_{L^p} + e^{c_0\nu^{-1}}\ee], \
    \text{for every $s \in [0, t]$}
    \big\} .
\end{equation}
Owing to Remarks~\ref{r} and~\ref{r2} and to \eqref{e:elleunoinfinito} and~\eqref{e:controlp}, the functions $u_{\ee \nu}$ and $u_\nu$ are continuous from  $[0, + \infty[$ to $L^p$. Hence, the function $\| z_\ee(t, \cdot)\|_{L^p}$ is  
continuous and $\tau>0$. Moreover, \eqref{e:deftau} 
implies 
\begin{equation}
\label{e:equality}
  \| z_\ee (\tau, \cdot) \|_{L^p } =  2 [\| z_0 \|_{L^p} + e^{c_0\nu^{-1}}\ee]. 
\end{equation}
We represent the solution of the Cauchy 
problem~\eqref{e:cpzeta} by the Duhamel Principle as 
\begin{equation}
\label{e:duhamel}
\begin{split}
       z_\ee(\tau, \cdot) & =
      G_\nu(\tau, \cdot) \ast z_0 - 
       \int_0^\tau \int_{\R^d} \nabla G_\nu(\tau-s, \cdot-y) \cdot
         \mathcal T_\ee (s, y)dy ds 
\end{split}       
\end{equation}
where $G_\nu$ denotes the heat kernel~\eqref{e:Gnu}. We apply \eqref{e:G1nu} and the Bochner and Young Theorems  to get 
		\begin{equation}
\label{e:iuno1}
\begin{split}
\| z_\ee(\tau, \cdot)\|_{L^p}
&{\leq} \|G_\nu(\tau, \cdot)\|_{L^1} \|z_0\|_{L^p} +\int_0^\tau 
\left\|\int_{\R^d} \nabla G_\nu(\tau-s, \cdot) \cdot  
\mathcal T_\ee (s, \cdot)
\right\|_{L^p} ds
\\&
{\leq} \|z_0\|_{L^p}+
\int_0^\tau \left\| \nabla G_\nu(\tau-s, \cdot) \right\|_{L^q}
\left\| 
\mathcal T_\ee (s, \cdot)
\right\|_{L^{p/2}} ds,
\end{split}
\end{equation}      
noting that $p/2 \ge 1$ owing to~\eqref{e:condizionip}, and setting $q:= p/(p-1)$. 
%By using the H\"older inequality, we get the following chain of inequalities. 
Only in the rest of this proof we denote by $C$ any constant that only depends on $d, p, L, \| \bar u \|_{L^\infty}, \| \bar u \|_{W^{1,p}}$. 
For every $s\in [0,\tau]$, by \eqref{e:cosaet}, the H\"older inequality and \eqref{e:v} we have 
\begin{equation}
%\label{e:iuno2}
\begin{split}
      \big\| 
       \mathcal T_\ee (s, \cdot) 
       \big\|_{L^{p/2}}
       & \leq \,
       \big(\| z_\ee \|_{L^p} + \| u_\nu \|_{L^p} \big)
      \left\| 
       b\big( [z_\ee + u_\nu]  \ast \eta_\ee\big)- b(u_\nu) 
       \right\|_{L^p} + \|z_\ee \|_{L^p} \| b(u_\nu )\|_{L^p}\\ 
       & \leq \,
       C  \big(\| z_\ee \|_{L^p} + \| u_\nu \|_{L^p} \big)
      \left\| 
       ( [z_\ee + u_\nu]  \ast \eta_\ee) - u_\nu
       \right\|_{L^p}  + C \|z_\ee \|_{L^p} \| u_\nu\|_{L^p}
\end{split}
\end{equation}
(all functions at the right hand side are evaluated at time $s$).
By the Young inequality we have $\| z_\ee \ast \eta_\ee \|_{L^p} \leq  \| z_\ee \|_{L^p}$, and applying also \eqref{e:elleduee} we get
\begin{equation}
\label{e:iuno2}
\begin{split}
		 \big\| 
		&  
		\mathcal T_\ee (s, \cdot) 
		\big\|_{L^{p/2}}{\leq}  
       C  \big(\| z_\ee(s, \cdot)  \|_{L^p} + \| u_\nu(s, \cdot)  \|_{L^p} \big)  \big[ 
        \| z_\ee (s, \cdot) \|_{L^p}  +
       e^{c_0\nu^{-1}} \ee \big] .
%        + C \|z_\ee \|_{L^p} \| u_\nu\|_{L^p}    
\end{split}
\end{equation}
We recall that $s \leq \tau$ and so~\eqref{e:ansatz} holds. Also, we recall~\eqref{e:alpha} and~\eqref{e:G1nu} and we point out that 
$\alpha = -(d+p)/(2p)\in( -1 , 0)$ by~\eqref{e:condizionip}.
Using this and~\eqref{e:iuno2}, we go back to~\eqref{e:iuno1} to get
\begin{equation}
\label{e:duhamel2}
\begin{split}
    \| z_\ee (\tau, \cdot) \|_{L^p} &\leq  \| z_0 \|_{L^p} + 
   C \nu^\alpha \big[ 2 [\| z_0 \|_{L^p} + e^{c_0\nu^{-1}}\ee] +1 \big] \big[ 2 [\| z_0 \|_{L^p} + e^{c_0\nu^{-1}}\ee] +  e^{c_0\nu^{-1}}\ee
    \big]  \tau^{\alpha+1}
    \\&\leq  \| z_0 \|_{L^p} + 
    6C \nu^\alpha [\| z_0 \|_{L^p} + e^{c_0\nu^{-1}}\ee] \tau^{\alpha+1},
       \end{split}
\end{equation}
where in the last inequality we used \eqref{hp:eps-z0} to show that $ 2 [\| z_0 \|_{L^p} + e^{c_0\nu^{-1}}\ee] +1  \leq 2$. By comparing~\eqref{e:duhamel2} with~\eqref{e:equality} we arrive at 
$$
  2 [\| z_0 \|_{L^p} + e^{c_0\nu^{-1}}\ee] \leq \| z_0 \|_{L^p} + 6C \nu^\alpha [\| z_0 \|_{L^p} + e^{c_0\nu^{-1}}\ee] \tau^{\alpha+1}, $$
  which implies 
 $$   [\| z_0 \|_{L^p} + e^{c_0\nu^{-1}}\ee] \leq 6C \nu^\alpha [\| z_0 \|_{L^p} + e^{c_0\nu^{-1}}\ee] \tau^{\alpha+1} \implies (6C)^{-1} \nu^{-\alpha}\leq \tau^{\alpha+1} .
$$
This gives a lower bound on $\tau$ and concludes the proof of the Lemma by choosing $\tau_0 =(6C)^{-\frac{1}{\alpha+1}}$. 
\end{proof}
	\subsubsection{Conclusion of the proof of Theorem~\ref{t:convreg}}
\label{sss:concludo}
Let $0<\nu<1/4$, and let $\tau_0$ and $c_0$ be as in the statement of Lemma~\ref{l:zeta}. Let
$
m: = \mathrm{int} \left( ({\tau_0 \nu^\beta})^{-1} \right) +1, 
$
where $\mathrm{int}(\cdot)$ denotes the integer part, and let $ \ee \leq  e^{-2c_0\nu^{-1}} 4^{-m-1}$. This is implied for instance by 
$$\ee \leq e^{-c \nu^{-\beta}} \qquad \mbox{for }c:= c(d, p, L, \| \bar u \|_{L^\infty},  \| \bar u \|_{W^{1,p}})>0.$$
We show by induction that for every $i=1,..., m$
\begin{equation}\label{eqn:induct}
\| z_\ee (t, \cdot) \|_{L^p} \leq 
e^{c_0\nu^{-1}} \ee 4^i
\quad \text{for every $t \in [(i-1)\tau_0 \nu^\beta, i\tau_0 \nu^\beta]$} .
\end{equation}
Indeed, by~\eqref{e:cpr},~\eqref{e:vclr} and~\eqref{e:zetaee} we have 
$z_\ee (0, x) \equiv 0$;
we apply estimate~\eqref{e:ansatz} on $[0, \tau_0 \nu^\beta]$ to get \eqref{eqn:induct} with $i=1$. If the statement holds true for $i$, we have that $\| z_\ee (i\tau_0 \nu^\beta, \cdot) \|_{L^p} \leq 
e^{c_0\nu^{-1}} \ee 4^i \leq e^{-c_0\nu^{-1}}/4 $; to get the statement for $i+1$, we apply 
estimate~\eqref{e:ansatz} with $z_0:= z_\ee(i\tau_0 \nu^\beta, \cdot)$, obtaining 
$$
\| z_\ee (t, \cdot) \|_{L^p } \leq 
2 e^{c_0\nu^{-1}} \big[\ee 4^i + \ee \big]\leq 4^{i+1}e^{c_0\nu^{-1}} \ee
\quad \text{for every $t \in [i\tau_0 \nu^\beta,(i+1)\tau_0 \nu^\beta].$} 
$$
This establishes~\eqref{e:convreg} and concludes the proof of Theorem~\ref{t:convreg}. \qed

\section{Proof of Theorem~\ref{t:convbadata}}
\label{s:teorema11}
We first explain the basic ideas of the proof. We need some notation: we term 
\begin{equation}
\label{e:semigroups}
S^{\ee \nu}_t: L^1 \cap L^\infty \times [0, +\infty[
\to L^1 \cap L^\infty, \qquad S^\nu_t: L^1 \cap L^\infty \times [0, +\infty[
\to L^1 \cap L^\infty
\end{equation} the semigroup of solutions of the equations at the first line of~\eqref{e:cpr} and~\eqref{e:vclr}, respectively. In other words, $u_{\ee \nu} (t, \cdot) = S^{\ee\nu}_t \bar u$ and $u_\nu(t, \cdot) = S^{\nu}_t \bar u$. Next, we fix $\dd \in L^1 \cap L^\infty $ and a regularity parameter $0<\lambda<1$ and  we decompose $\dd$ as 
\begin{equation}
\label{e:di}
    \dd : ={\displaystyle{\dd_r}} +{\displaystyle{\dd_s}}, \qquad \mbox{where}\qquad {\displaystyle{\dd_r}}:={\dd \ast \rho_\lambda}, \qquad
    {\displaystyle{\dd_s}}={ \dd - \dd\ast \rho_\lambda}. 
\end{equation}
In the previous expression $\rho_\lambda$ is a given standard family of convolution kernels, obtained by setting 
$\rho_\lambda (x) : = \lambda^{-d} \rho \left(  x / \lambda \right)$ for a standard (i.e., smooth, positive, radial, compactly supported, and with unit integral) convolution kernel $\rho$, with $\| \rho\|_{C^1(\R^d)} \leq C(d)$. 

%
%. Since $\rho$ is fixed, all its norms only depend on the space di
%
%
%we  $\rho$ to be fixed, in all the following estimates we will not track the dependence on $\rho$, but only on the regularity parameter $\lambda$ when appropriate.

Note that $\dd_r$ is regular, and hence we can apply Theorem~\ref{t:convreg} to show that $S^{\ee \nu}_t \dd_r$
converges to $S^\nu_t \dd_r$, with a convergence rate that deteriorates when $\lambda \to 0^+$. Also, we can choose the regularizing parameter $\lambda$ in such a way that $\dd_s=\dd-\dd_r$ is small. The basic point in the proof of Theorem~\ref{t:convbadata} is then establishing a uniform control 
on the growth of  $\| S^{\ee \nu}_t \dd- S^{\ee \nu}_t \dd_r  \|_{L^p}$. This is done in~\S~\ref{ss:perturbo} below. Next, in~\S~\ref{ss:stability} we establish some stability estimates and in~\S~\ref{ss:concludiamo} we conclude the proof by using an iteration argument.  
\subsection{Perturbations estimates} We begin by establishing some perturbation estimates.
\label{ss:perturbo}
%The main result of this section is the following $L^p$ perturbation estimate.  
\begin{lemma}
\label{l:duhamel1}
Fix $p$ satisfying~\eqref{e:condizionip}, $\dd \in L^1\cap L^\infty$, and let $\dd_r$ and $\dd_s$ be as in~\eqref{e:di}. 
Assume that 
\begin{equation}
\label{e:delta}
    \| \dd_s\|_{L^p} \leq \delta \leq 1,
\end{equation}
 \begin{equation}
 \label{e:digrande}
   \| \dd \|_{L^p} \leq D, \quad \| \dd\|_{L^\infty} \leq B
 \end{equation}
for some positive constants $D>0$, $B>0$. 
Then there are constants $ \bar \ee (d, p, L, B, D,\nu, \lambda)$
and  $\sigma = \sigma (d, p, L, D, \nu)$ such that, if $\ee \leq \bar \ee$, then 
\begin{equation}
\label{e:zanzara}
     \| S^{\ee \nu}_t \dd - S^{\ee \nu}_t \dd_r \|_{L^p} \leq 2 \delta, 
     \quad \text{for every $t \in [0, \sigma].$}
\end{equation}
\end{lemma}
\begin{proof}
We set
\begin{equation}
\label{e:cosaev}
      v_\ee : = S^{\ee \nu}_t \dd - S^{\ee \nu }_t \dd_r 
\end{equation}
and we point out  that $v_\ee$ is a solution of the Cauchy problem
\begin{equation*}
\left\{
\begin{array}{ll}
\partial_t v_\ee + \div \big[ v_\ee b(  S^{\ee \nu}_t \dd \ast \eta_\ee) 
+ S^{\ee \nu}_t \dd_r \big( b(  S^{\ee \nu}_t \dd \ast \eta_\ee)  -
b(  S^{\ee \nu}_t \dd_r \ast \eta_\ee) 
\big) 
\big] = \nu \Delta v_\ee \\
v_\ee (0, x) = \dd_s (x). \\
\end{array}
\right.
\end{equation*}
We introduce $\sigma$ by setting 
$$
    \sigma : = \sup \big\{ t \in [0, 1]: \; \| v_\ee (s, \cdot) \|_{L^p} 
    \leq 2 \delta \; \text{for every 
     $s \in [0, t]$} \big\} .
$$
Note that, if $\sigma<1$, we have
\begin{equation}
\label{e:uguale}
    \| v_\ee (\sigma, \cdot) \|_{L^p} = 2  \delta. 
\end{equation}
We now provide a lower bound on $\sigma$. 
By using the Duhamel representation formula we get
\begin{equation}
\label{e:mestre}
     v_\ee (t, \cdot) = G_\nu(t, \cdot) \ast \dd_s - 
       \int_0^t \int_{\R^d}
       \nabla G_\nu(t-s, \cdot-y) \cdot \Big[ v_\ee b(  S^{\ee \nu}_t \dd \ast \eta_\ee) 
+ S^{\ee \nu}_t \dd_r \big[ b(  S^{\ee \nu}_t \dd \ast \eta_\ee)  -
b(  S^{\ee \nu}_t \dd_r \ast \eta_\ee) 
\big] 
\Big] (s, y) dy ds. 
\end{equation}
We fix $q:= p/(p-1)$ and $ \alpha > -1$ given by~\eqref{e:alpha}. Applying the Bochner and Young Theorems we get
\begin{equation}
\begin{split}
\label{e:mestre2}
        \| v_\ee (t, \cdot)\|_{L^p} 
       & \leq 
        \| G_\nu(t, \cdot)\|_{L^1} \| \dd_s \|_{L^p} \\
        & \quad + 
       \int_0^t \| \nabla 
       G_\nu(t-s, \cdot) \|_{L^q}
       \Big[  \| v_\ee b(  S^\ee_t \dd \ast \eta_\ee) (s, \cdot) +  S^\ee_t \dd_r \big[ b(  S^\ee_t \dd \ast \eta_\ee)  -
b(  S^\ee_t \dd_r \ast \eta_\ee) \big] (s, \cdot) \|_{L^{p/2}}  \Big] ds
\\  & \hspace{-0.5cm} \stackrel{\eqref{e:G1nu},\eqref{e:delta}}{\leq} 
     \delta  + C(d, p, \nu)
       \int_0^t (t-s)^{\alpha}
       \Big[  \| v_\ee b(  S^{\ee \nu}_t \dd \ast \eta_\ee) (s, \cdot) \|_{L^{p/2}}\\
       & \hspace{6.5cm}
+ \| S^{\ee \nu}_t \dd_r \big[ b(  S^{\ee \nu}_t \dd \ast \eta_\ee)  -
b(  S^{\ee \nu}_t \dd_r \ast \eta_\ee) \big] (s, \cdot) \|_{L^{p/2}}  \Big] ds .
\end{split}
\end{equation}
Next, by the H\"older inequality, \eqref{e:v},~\eqref{e:cosaev} and the Young inequality to get 
\begin{equation}
\label{e:mestre3}
\begin{split}  
    \| v_\ee b(  S^{\ee \nu}_t \dd &\ast \eta_\ee) \|_{L^{p/2}}
+ \| S^{\ee \nu}_t \dd_r \big[ b(  S^{\ee \nu}_t \dd \ast \eta_\ee)  -
b(  S^{\ee \nu}_t \dd_r \ast \eta_\ee) \big]  \|_{L^{p/2}}
    \\
    & \leq 
     \| v_\ee \|_{L^p} \| b(  S^{\ee \nu}_t \dd \ast \eta_\ee)\|_{L^p}
+ \| S^{\ee \nu}_t \dd_r \|_{L^p} \| b(  S^{\ee \nu}_t \dd \ast \eta_\ee)  -
b(  S^{\ee \nu}_t \dd_r \ast \eta_\ee) \|_{L^p}  
    \\  & \leq 
     L \| v_\ee \|_{L^p} \| S^{\ee \nu}_t \dd \ast \eta_\ee\|_{L^p}
+L \| S^{\ee \nu}_t \dd_r \|_{L^p} \| v_\ee \ast \eta_\ee \|_{L^p}  \\
 & \leq 
  L \| v_\ee \|_{L^p} \| S^{\ee \nu}_t \dd \|_{L^p}
       +L \| S^{\ee \nu}_t \dd_r \|_{L^p} \| v_\ee \|_{L^p}  \\
       & \leq L
  \| v_\ee \|_{L^p} \big[ \| S^{\ee \nu}_t \dd_r \|_{L^p} + \| v_\ee \|_{L^p}
       \big] 
       + L\| S^{\ee \nu}_t \dd_r \|_{L^p} \| v_\ee \|_{L^p}  \\
       & \leq 2 L 
        \| v_\ee \|_{L^p} \big[ \| S^{\ee \nu}_t \dd_r \|_{L^p} + \| v_\ee \|_{L^p}
       \big] . \\
\end{split} 
\end{equation}
We recall the definition~\eqref{e:di} of $\dd_r$ and we point out that $\dd_r$ is smooth and henceforth satisfies the hypotheses of Theorem~\ref{t:convreg}. By applying~\eqref{e:convreg} and \eqref{e:controlp}, we get 
\begin{equation}
\label{e:regolarizzata} 
\begin{split}
        \| S^{\ee \nu}_t \dd_r \|_{L^p} \leq 
        \| S^{\ee \nu}_t \dd_r - S^\nu_t \dd_r  \|_{L^p}
        + \| S^\nu_t \dd_r \|_{L^p} {\leq} 
       C(d, p, L, \| \dd_r \|_{L^\infty},  \| \dd_r \|_{W^{1,p}}, \nu)  \ee     
        + \| \dd_r \|_{L^p} .
\end{split}   
\end{equation} 
Since $\dd_r = \dd \ast \rho_\lambda$, by \eqref{e:digrande} we have 
\begin{equation}
\label{e:stimeconv}
    \| \dd_r \|_{L^\infty} \leq \| \dd \|_{L^\infty} {\leq} B, 
    \quad 
     \| \dd_r \|_{L^p} \leq \| \dd \|_{L^p} {\leq} D, 
     \quad 
    \| \nabla \dd_r \|_{L^p} \leq \| \dd \|_{L^p} \| \nabla \rho_\lambda \|_{L^1}
    {\leq} C(d,D, \lambda)  
\end{equation}
and hence~\eqref{e:regolarizzata} implies
 $$
        \| S^{\ee \nu}_t \dd_r \|_{L^p} \leq 
        C(d, p, L, B,  D,\nu, \lambda)  \ee     
        + D,
$$ 
so if $\ee \leq \bar \ee (d, p, L, B,  D,\nu, \lambda)$ is sufficiently small, then 
\begin{equation}
\label{e:regolarizzata2} 
          \| S^{\ee \nu}_t \dd_r \|_{L^p} \leq 
         {3}D/2.
\end{equation}
We combine~\eqref{e:uguale},~\eqref{e:mestre2},~\eqref{e:mestre3} and~\eqref{e:regolarizzata2} and we get that
\begin{equation}
\begin{split}
     \| v_\ee (\sigma, \cdot) \|_{L^p} & \;= 2 \delta \leq \delta  + C(d, L, p, \nu)
       \int_0^{\sigma} (\sigma-s)^\alpha
       \Big[  \| v_\ee \|_{L^p} \big[ D + \| v_\ee \|_{L^p}
       \big] \Big]
       (s, \cdot) ds \\
       & \stackrel{s \leq \sigma}{\leq}
        \delta  + C(d, L, p, \nu)
       \int_0^{\sigma} (\sigma-s)^\alpha
       \Big[  2 \delta \big[ D + 2 \delta
       \big] \Big]
       (s, \cdot) ds
       = \Big[ 1 + C(d, L, p, \nu)  \sigma^{\alpha+1}\big[ D + 2 \delta  \big] \Big] \delta  \\
       & \stackrel{\delta \leq 1}{\leq}
        \Big[ 1 + C(d, L, p, \nu)  \sigma^{\alpha +1 }\big[ D + 2 \big] \Big] \delta.      
\end{split}
\end{equation}
The above chain of inequalities implies that 
 $
     1 \leq  
     C(d, L, p, \nu)  \sigma^{\alpha +1 }\big[ D + 2 \big]
 $
and this provides a lower bound on $\sigma$ that only depends on $d$, $p$, $L$, $\nu$ and $D$. 
\end{proof} 
\subsection{Stability estimates}
\label{ss:stability}
We now establish a conditional stability estimate. 
\begin{lemma}
\label{l:estability}
Fix $\dd_1, \dd_2 \in L^1  \cap L^\infty $ and $p$ satisfying~\eqref{e:condizionip} and assume there are constants $F>0$ and $T>0$ such that 
\begin{equation}
\label{e:hypstab}
       \| S^{\ee \nu}_t \dd_1 \|_{L^p}, \;  
       \| S^{\ee \nu}_t \dd_2 \|_{L^p}
       \leq F, \quad \text{for every $t \in [0, T]$.}
\end{equation}
Then there is a threshold $\varpi = \varpi(L, F, d, p, \nu) \in \, ]0,T]$ such that 
\begin{equation}
\label{e:tstab}
          \| S^{\ee \nu}_t \dd_1 -  S^{\ee \nu}_t \dd_2 \|_{L^p} \leq 
           2 \| \dd_1 - \dd_2 \|_{L^p} , 
          \quad \text{for every $t \in [0, \varpi]$.}
\end{equation}
\end{lemma}
\begin{proof}
We use the Duhamel Representation Formula and get 
\begin{equation*}
\begin{split}
  S^{\ee \nu}_t \dd_1 - S^{\ee \nu}_t \dd_2 & =
  [\dd_1 - \dd_2] \ast G_\nu(t, \cdot) \\
  & \qquad -
  \int_0^t \int_{\R^d} 
  \nabla G_\nu (t-s, \cdot-y)  
  \big[S^{\ee \nu}_s \dd_1 \, b(S^{\ee \nu}_s \dd_1 \ast \eta_\ee)
  -  S^{\ee \nu}_s \dd_2 \, b(S^{\ee \nu}_s \dd_2 \ast \eta_\ee)
\big] (y) dyds,
\end{split}
\end{equation*}
which owing to the Bochner and Young Theorems implies 
\begin{equation}
\label{e:pezzo0}
\begin{split}
  \| S^{\ee \nu}_t \dd_1 - S^{\ee \nu}_t \dd_2\|_{L^p} & 
\leq 
  \| \dd_1 - \dd_2 \|_{L^p} \| G_\nu(t, \cdot) \|_{L^1}
  \\ & +
  \int_0^t \| 
  \nabla G_\nu (t-s, \cdot) \|_{L^q}
 \| S^{\ee \nu}_s \dd_1 \, b(S^{\ee \nu}_s \dd_1 \ast \eta_\ee)
  -  S^{\ee \nu}_s \dd_2 \, b(S^{\ee \nu}_s \dd_2 \ast \eta_\ee)\|_{L^{p/2}} ds 
  \end{split}
\end{equation}  
provided $q:= p/(p-1)$. 
%By using~\eqref{e:G1nu}, we get 
%\begin{equation}
%\label{e:pezzo1}
%   \| d_1 - d_2 \|_{L^p} \| G_\nu(t, \cdot) \|_{L^1}
%   \leq  \| d_1 - d_2 \|_{L^p}. 
%\end{equation}
By the H\"older inequality we get 
\begin{equation}
\label{e:pezzo3} 
\begin{split}
    \| S^{\ee \nu}_s & \dd_1 \, b(S^{\ee \nu}_s \dd_1 \ast \eta_\ee)
  -  S^{\ee \nu}_s \dd_2 \, b(S^{\ee \nu}_s \dd_2 \ast \eta_\ee)\|_{L^{p/2}} 
  \\ &\leq 
  \| S^{\ee \nu}_s \dd_1 \, b(S^{\ee \nu}_s \dd_1 \ast \eta_\ee)
  - S^{\ee \nu}_s \dd_2 \, b(S^{\ee \nu}_s \dd_1 \ast \eta_\ee)\|_{L^{p/2}}
  + \|S^{\ee \nu}_s \dd_2 \, b(S^{\ee \nu}_s \dd_1 \ast \eta_\ee) 
   -  S^{\ee \nu}_s \dd_2 \, b(S^{\ee \nu}_s \dd_2 \ast \eta_\ee)\|_{L^{p/2}}
  \phantom{\int}
  \\
  & \hspace{-0.7cm} \stackrel{\text{H\"older},\eqref{e:v}}{\leq} 
   L \big\|  S^{\ee \nu}_s \dd_1 -  S^{\ee \nu}_s \dd_2 \big\|_{L^p}\, 
    \big\| S^{\ee \nu}_s \dd_1 \ast \eta_\ee \big\|_{L^p} +
      L  \big\|  S^{\ee \nu}_s \dd_2 \big\|_{L^p} 
      \big\| \big[S^{\ee \nu}_s \dd_1 
  -    S^{\ee \nu}_s \dd_2 \big]\ast \eta_\ee \big\|_{L^p}   \\
  & \hspace{-0.6cm} \stackrel{\text{Young}, \eqref{e:eta}}{\leq} L
   \big\|  S^{\ee \nu}_s \dd_1 -  S^{\ee \nu}_s \dd_2 
   \big\|_{L^p}\, 
    \big\| S^{\ee \nu}_s \dd_1 \big\|_{L^p}  +
    L \big\|  S^{\ee \nu}_s \dd_2 \big\|_{L^p} 
      \big\| S^{\ee \nu}_s \dd_1 
  -    S^\ee_s \dd_2  \big\|_{L^p}   \\
  & \hspace{-0.2cm} \stackrel{\eqref{e:hypstab}}{\leq}
   C(L, F)
   \big\|  S^{\ee \nu}_s \dd_1 -  S^{\ee \nu}_s \dd_2 \big\|_{L^p}. 
\end{split}
\end{equation}
We now introduce the value $\varpi$ by setting 
\begin{equation}
\label{e:varpi}
\varpi : = \sup \big\{ t \in [0, 1]: 
\;  \big\|  S^{\ee \nu}_s \dd_1 -  S^{\ee \nu}_s \dd_2 \big\|_{L^p} 
\leq 2  \big\|  \dd_1 -  \dd_2 \big\|_{L^p}
\; \text{for every $s \in [0, t]$} \big\}. 
\end{equation}
Note that 
\begin{equation}
\label{e:pezzo4}
  \big\|  S^{\ee \nu}_\varpi \dd_1 -  S^{\ee \nu}_\varpi \dd_2 \big\|_{L^p} 
=2  \big\|  \dd_1 -  \dd_2 \big\|_{L^p}.
\end{equation}
Also, by combining~\eqref{e:pezzo0},~\eqref{e:G1nu} and~\eqref{e:pezzo3} we get that 
\begin{equation}
\label{e:pezzo5}
\begin{split}
  \big\|  S^{\ee \nu}_\varpi \dd_1 -  S^{\ee \nu}_\varpi \dd_2 \big\|_{L^p} 
  & \leq  \big\|  \dd_1 -  \dd_2 \big\|_{L^p}
   +
  C(L, F) \big\|  \dd_1 -  \dd_2 \big\|_{L^p}
  \int_0^\varpi
  \| 
  \nabla G_\nu (\varpi-s, \cdot) \|_{L^q} 
  ds \\
  & \hspace{-0.5cm}\stackrel{\eqref{e:G1nu},\eqref{e:alpha}}{\leq}
  \big\|  \dd_1 -  \dd_2 \big\|_{L^p}
   +
  C(L, F, p, d, \nu) \big\|  \dd_1 -  \dd_2 \big\|_{L^p}
  \int_0^\varpi (\varpi -s)^\alpha  
  ds
  \\ & \leq  
  \big\|  \dd_1 -  \dd_2 \big\|_{L^p}
  [ 1 
   +
  C(L, F, p, d, \nu) \varpi^{\alpha +1}
  ]. 
  \end{split}
\end{equation}
By comparing~\eqref{e:pezzo5} with~\eqref{e:pezzo4} we get
$$
   2 
   \big\|  \dd_1 -  \dd_2 \big\|_{L^p} \leq \big[1 + 
   C(L, F, d, p, \nu)
    \varpi^{\alpha +1}\big] 
    \big\|  \dd_1 -  \dd_2 \big\|_{L^p}
$$
and this provides a lower bound on $\varpi$. 
 \end{proof}
We conclude this paragraph by establishing a uniform a-priori estimate on the growth of $\dd$. 
\begin{lemma}
\label{l:lpcontrol}
Assume that $\dd \in L^\infty \cap L^1$ and that 
\begin{equation}
\label{e:ellepdi}
  \| \dd \|_{L^p} \leq Q. 
\end{equation}
Then there is a constant $\theta = \theta (d, p,L, Q)>0$ such that 
\begin{equation}
\label{e:controlpdi}
  \| S^{\ee \nu}_t \dd \|_{L^p}
  \leq 2 Q, \quad \text{for every $t \in [0, \theta]$.}
\end{equation}
\begin{proof}
We set
$$
   \theta: = \sup \big\{ t \in [0, 1 ]: \| S^{\ee \nu}_s \dd \|_{L^p}
  \leq 2 Q, \; \text{for every $s \in [0, t]$} \big\}
$$
and we point out that 
\begin{equation}
\label{e:eqtheta}
     \| S^{\ee \nu}_\theta \dd \|_{L^p}
  = 2 Q. 
\end{equation}
To establish a lower bound on $\theta$ we use the Duhamel representation formula. We have 
\begin{equation*}
\begin{split}
  S^{\ee \nu}_\theta \dd   =
  \dd \ast G_\nu(\theta, \cdot) 
   -
  \int_0^\theta \int_{\R^d} 
  \nabla G_\nu (\theta-s, \cdot-y)  \cdot 
  \big[S^{\ee \nu}_s \! \dd  \ b(S^{\ee \nu}_s \dd \ast \eta_\ee)
\big] (y) dyds.
\end{split}
\end{equation*}
We use the Bochner  and  Young Theorems  and we get
$$
  \| S^{\ee \nu}_\theta \dd \|_{L^p} \leq 
  \| \dd  \|_{L^p} 
   +
  \int_0^\theta \|  
  \nabla G_\nu (\theta-s, \cdot)  \|_{L^q}
  \| S^{\ee \nu}_s \! \dd  \ b(S^{\ee \nu}_s \dd \ast \eta_\ee)\|_{L^{p/2}} ds, 
$$
provided $q:= p/(p-1)$. Next, by H\"older inequality, \eqref{e:v}, and since $s \leq \theta$ we infer that 
\begin{equation*}
\begin{split}
     \| S^{\ee \nu}_s \! \dd   \ b(S^{\ee \nu}_s \dd \ast \eta_\ee)\|_{L^{p/2}}
     & \leq 
      \| S^{\ee \nu}_s \! \dd  \|_{L^p} \| b(S^{\ee \nu}_s \dd \ast \eta_\ee)\|_{L^p}
     {\leq}
     L \| S^{\ee \nu}_s \! \dd  \|_{L^p}   
      \|S^{\ee \nu}_s \dd \ast \eta_\ee\|_{L^p}
      \leq  L \| S^{\ee \nu}_s \! \dd  \|^2_{L^p} {\leq}
      4 L Q^2. 
      \end{split}
\end{equation*}
We let $ \alpha > -1$ be as in~\eqref{e:alpha}. By \eqref{e:G1nu} and the above inequalities 
we infer that 
$$
     \| S^{\ee \nu}_\theta \dd \|_{L^p} \leq 
     Q + C(d,p,L) \theta^{\alpha+1} Q^2
$$
and by comparing the above inequality with~\eqref{e:eqtheta} we establish a lower bound on $\theta$. 
\end{proof}
\end{lemma}

\subsection{Conclusion of the proof of Theorem~\ref{t:convbadata}} 
\label{ss:concludiamo}
We first introduce some notation. First, we fix a parameter $0 < h < 1$.  We set 
\begin{equation}
\label{e:parametri}
    D: =  \| \bar u \|_{L^p}, \quad B := \| \bar u \|_{L^\infty}, \quad F := 4 D, \quad Q:= 2 D
\end{equation}
and choose a threshold $\xi = \xi(d, p, L, D, \nu, Q)$ in such a way that 
$$
    \xi : = \min \{ \sigma, \varpi, \theta \},
$$
where $\sigma $, $\varpi$ and $\theta$ are as in the statement of Lemma~\ref{l:duhamel1}, Lemma~\ref{l:estability} and Lemma~\ref{l:lpcontrol}, respectively. \\
{\sc Step 1:} we choose $\dd:= \bar u$ and the regularity parameter $\lambda$ in~\eqref{e:di} (depending only on on $p$, $\bar u$ and $h$) in such a way that 
$$
\| \dd_s \|_{L^p} =   \| \dd - \dd \ast \rho_\lambda \|_{L^p}\leq h <1.
$$
We establish convergence on the interval $[0, \xi]$. 
First we decompose $\bar u$ as in~\eqref{e:di}.  
Note that 
$$
   \| \dd_{r}\|_{L^p} \leq \| \dd \|_{L^p} \leq D, \quad 
   \| \dd_r  \|_{L^\infty} \leq \| \dd \|_{L^\infty}\leq B. 
$$
Next, we fix $t \in [0, \xi]$ and we introduce the following decomposition:
\begin{equation}
\label{e:decompo1}
\begin{split}
      \| S^{\ee \nu}_t \dd - S^{\nu}_t \dd \|_{L^p} 
      & \leq 
      {\| S^{\ee \nu}_t \dd - S^{\ee \nu}_t \dd_r \|_{L^p}}
       + 
       \| S^{\ee \nu}_t \dd_{r }- S^{ \nu}_t \dd_{ r} \|_{L^p} 
        +  \| S^{\nu}_t \dd_{r} - S^{ \nu}_t \dd \|_{L^p} =: {T_1} +{T_2} +{T_3}.
       \end{split} 
\end{equation}
To control the term $T_1$, we apply Lemma~\ref{l:duhamel1} and we infer that, if $\ee \leq 
\bar \ee (d, p, L, B,D, \nu, \lambda)$, then 
\begin{equation}
\label{e:tiuno}
      \| S^{\ee \nu}_t \dd - S^{\ee \nu}_t \dd_r  \|_{L^p} \leq 2 h. \end{equation}
To control the term $T_2$, we apply Theorem~\ref{t:convreg}. First, we point out that
$$
   \nabla \dd_r = \dd \ast \nabla \rho_\lambda \implies 
   \|  \nabla \dd_r \|_{L^p}
   \leq \| \bar u \|_{L^p} 
   \| \nabla \rho_\lambda \|_{L^1} = C(d,D, \lambda). 
$$
%where the constant $K$ is defined by setting %{\color{red} possiamo togliere questa dipendenza scegliendo un $\eta$ specifico, no?}
%\begin{equation}
%\label{e:kappa2}
%    K : = \| \nabla \rho \|_{L^1}. 
%\end{equation} 
By applying Theorem~\ref{t:convreg} we arrive at 
\begin{equation}
\label{e:tidue}
  \| S^{\ee \nu}_t \dd_r - S^{ \nu}_t \dd_r \|_{L^p}
  \leq C (d, p, L, B, D, \nu,\lambda) \ee
  \leq h
\end{equation} 
provided that $\ee \leq \bar \ee (d, p, L, B, D, \lambda, \nu, h)$. 
Finally, to control the term $T_3$ we apply~\eqref{e:stability} and we get 
\begin{equation}
\label{e:titre}
\| S^{\nu}_t \dd_r - S^{ \nu}_t \dd \|_{L^p}
\leq C(d, p, L, B, \nu) 
\| \dd_r  - \dd \|_{L^p}
\leq C(d, p, L,B, \nu)  h. 
\end{equation}
By combining~\eqref{e:tiuno}, \eqref{e:tidue} and~\eqref{e:titre} with~\eqref{e:decompo1} we eventually get that 
 \begin{equation}
 \label{e:concludopasso2}
      \| S^{\ee \nu}_t \dd - S^{\nu}_t \dd \|_{L^p} 
       \leq 
      C(d, p, L, B, \nu)  h
\end{equation}
provided that $\ee \leq \bar \ee (d, p, L, B, D, \lambda, \nu, h)$. \\
{\sc Step 2:} we establish convergence on the interval $[\xi, 2 \xi]$. First, we fix $t \in [0, \xi]$ and we introduce the following decomposition:
\begin{equation}
\label{e:decompo2}
\begin{split}
\| S^{\ee \nu}_{t+\xi} \bar u - S^{\nu}_{t+\xi} \bar u 
\|_{L^p} & =
\| S^{\ee \nu}_{t}S^{\ee \nu}_{\xi} \bar u 
- S^{\nu}_{t} S^\nu_\xi \bar u 
\|_{L^p} \\ &
       \leq 
      \| S^{\ee \nu}_{t}S^{\ee \nu}_{\xi} \bar u 
      - S^{\ee \nu}_{t} S^{\nu}_{\xi} \bar u  \|_{L^p} 
       + 
       \| S^{\ee \nu}_{t}S^{\nu}_{\xi} \bar u -
        S^{ \nu}_{t}S^{\nu}_{\xi} \bar u  \|_{L^p} =:{S_1}+{S_2} .
\end{split}
\end{equation}
To control the term $S_1$ we apply Lemma~\ref{l:estability}. 
First, we set 
$$
    \dd_1 := S^{\ee \nu}_{\xi} \bar u, \qquad 
    \dd_2 : = S^{\nu}_{\xi} \bar u
$$
and we recall that $F = 4 \| \bar u \|_{L^p}$.
Now we want to show that~\eqref{e:hypstab} holds true: we do this by applying 
Lemma~\ref{l:lpcontrol}. First, we check that 
\begin{equation}
\label{e:primocontrollo}
   \| S^{\ee \nu}_t \dd_2 \|_{L^p} \leq F. 
\end{equation}
We recall that $Q= 2 \| \bar u \|_{L^p}$ and we point out that, owing to~\eqref{e:controlp}, we have 
$$
  \| \dd_2 \|_{L^p} \leq \| \bar u \|_{L^p}\leq Q.
$$
By applying Lemma~\ref{l:lpcontrol}, we get~\eqref{e:primocontrollo}. 
Next, by \eqref{e:concludopasso2} and \eqref{e:controlp} we point out that 
$$
   \| \dd_1 \|_{L^p} \leq \|  S^{\ee \nu}_{\xi} \bar u - S^{\nu}_{\xi} \bar u 
   \|_{L^p} + \| S^{\nu}_{\xi} \bar u \|_{L^p} {\leq}
   C(d, p, L, B, \nu)  h + \| \bar u \|_{L^p}
   \leq 2  \| \bar u \|_{L^p} = Q
$$
provided that $h$ is sufficiently small. By applying Lemma~\ref{l:lpcontrol}, we get 
$ \| S^{\ee \nu}_t \dd_1 \|_{L^p} \leq F$
and by recalling~\eqref{e:primocontrollo} we conclude that~\eqref{e:hypstab} is satisfied. By applying Lemma~\ref{l:estability} we conclude that 
$$
   S_1\stackrel{\eqref{e:decompo2}}{=}\| S^{\ee \nu}_{t}S^{\ee \nu}_{\xi} \bar u 
      - S^{\ee \nu}_{t} S^{\nu}_{\xi} \bar u  \|_{L^p} \leq 
   2 \| S^{\ee \nu}_{\xi} \bar u 
      - S^{\nu}_{\xi} \bar u  \|_{L^p}  
       \stackrel{\eqref{e:concludopasso2}}{\leq}
       C(d, p, L, B, \nu)  h,  
$$
provided that $\ee \leq \bar \ee (d, p, L, B, D, \lambda, \nu, h)$. 

We now control $S_2$, the second term in~\eqref{e:decompo2}. We set
$\dd: = S^\nu_\xi \bar u$
and we point out that 
$$
   \| S^\nu_\xi \bar u \|_{L^p} \leq D, \quad
  \| S^\nu_\xi \bar u \|_{L^\infty} \leq B
$$
owing to~\eqref{e:controlp}. By applying the same argument as in {\sc Step 1} we conclude that 
$$
   S_2 {=} \| S^{\ee \nu}_{t}S^{\nu}_{\xi} \bar u -
        S^{ \nu}_{t}S^{\nu}_{\xi} \bar u  \|_{L^p} \leq 
     C(d, p, L, B, \nu)  h   
$$
provided that $\ee \leq \bar \ee (d, p, L, B, D, \lambda, \nu, h)$. 
By recalling~\eqref{e:decompo2}, this establishes the convergence on the interval $[\xi, 2 \xi]$. \\
{\sc Step 3:} by iterating the argument at {\sc Step 2} a finite number of times we can prove that 
$$
  \| S^{\ee \nu}_{t} \bar u - S^{\nu}_{t} \bar u 
\|_{L^p} \leq C(d, p, L, B, \nu)  h, \quad \text{for every $t \in [0, 1]$}  
$$
provided that $\ee \leq \bar \ee (d, p, L, B, D, \lambda,  \nu, h)$. This establishes the strong $L^p$ convergence and concludes the proof of Theorem~\ref{t:convbadata}. \qed
\section{Counterexamples}
\label{s:ce}
In this section we focus on the family of Cauchy problems in one space dimension 
\begin{equation}
\label{e:cpce}
\left\{
\begin{array}{ll}
\partial_t u_\ee + \partial_x \big[ u_\ee \, u_\ee \ast \eta_\ee\big] =0 \\
u_\ee (0, \cdot)= \bar u,
\end{array}
\right.
\end{equation}  
which is exactly~\eqref{e:nlcpr} in the case when $d=1$ and  $b(u) =u$. When $\ee \to 0^+$, the Cauchy problem in~\eqref{e:cpce} \emph{formally} reduces to the Cauchy problem for the Burgers' equation 
\begin{equation}
\label{e:cpburgers}
\left\{
\begin{array}{ll}
\partial_t u + \partial_x \big[ u^2 \big] =0 \\
u (0, \cdot)= \bar u.  
\end{array}
\right.
\end{equation} 
In this section we provide three explicit counterexamples showing that, in general, $u_\ee$ does not converge to the entropy admissible solution $u$. 
\subsection{A counterexample with sign-changing data and symmetric kernels} We begin by stating and proving our first counterexample. 
\label{ss:ce1}
%This paragraph aims at establishing the following lemma. 
\begin{counterexample}
\label{l:ce1}
Assume that $\eta_\ee$ satisfies~\eqref{e:eta} and~\eqref{e:etaee} and that $\eta$ is an even function, namely $\eta(x) = \eta(-x)$ for every $x$.
Assume furthermore that the initial datum $\bar u  \in BV(\R)$ is an odd function, namely $\bar u(x) = - \bar u(-x)$ for a.e.~$x$, and such that
\begin{equation}
\label{e:uzero1}
    \bar u(x) : = 
    \left\{
    \begin{array}{ll}
    1 & -1 < x <0 \\
    -1 & 0 < x < 1 \\
    0 & |x| > 2. \\ 
    \end{array}
    \right. 
 \end{equation}
 Let $u_\ee$ be the solution of~\eqref{e:cpce} and $u$ be the entropy admissible solution of~\eqref{e:cpburgers}.
 Then
 \begin{equation}
\label{e:ideace12-new}
    \int_{-\infty}^{0} u (t, x) dx <   \int_{-\infty}^{0} \bar u(x) dx = \int_{-\infty}^{0} u_\ee (t, x) dx , 
    \quad \text{for every $t \in [0, 1/4[$.}
\end{equation}
In particular, the family $\{ u_\ee \}_{\ee >0}$ does not converge to $u$, not even in the weak topology of $L^p$, $p \geq 1$, in the weak$^*$ topology of $L^\infty$, or up to subsequences.
\end{counterexample}  

The precise meaning of the last statement is the following: for every $p \ge 1$ and $T>0$ the statement ``there is a sequence $\ee_k$ such that $\ee_k \to 0^+$  and $u_{\ee_k} \rightharpoonup u$ in $L^p ([0, T] \times \R)$" is false; the statement ``there is a sequence $\ee_k$ such that $\ee_k \to 0^+$  and  $u_{\ee_k} \weaks u$ in $L^\infty ([0, T] \times \R)$" is also false.
The basic idea underpinning Counterexample~\ref{l:ce1} is, very loosely speaking, the following: one can show that for $t$ small enough, the entropy admissible solution of the Cauchy problem~\eqref{e:cpburgers},~\eqref{e:uzero1} has a steady shock at $x=0$ between the values $1$ (on the left) and $-1$ (on the right). By using the formal computation 
$$
 \frac{d}{dt} \int_{-\infty}^{0} u (t, x) dx \stackrel{\eqref{e:cpburgers}}{=}
     - \int_{-\infty}^{0} \partial_{x} [u^2] (t, x)dx =
          - u^2 (0^-) = -1 <0
$$
we infer the first inequality in \eqref{e:ideace12-new}. On the other hand, we can show that the solution $u_\ee$ of~\eqref{e:cpce},~\eqref{e:uzero1} is odd. Since the function $\eta_\ee$ is even, this implies that $u_\ee \ast \eta_\ee =0$ at $x=0$ and hence that 
$$
   \frac{d}{dt} \int_{-\infty}^{0} u_\ee (t, x) dx \stackrel{\eqref{e:cpce}}{=}
     - \int_{-\infty}^{0} \partial_{x} [u_\ee \ u_\ee \ast \eta_\ee ] (t, x)dx =0, 
     $$
which in turn implies the equality in \eqref{e:ideace12-new}.
By~\eqref{e:ideace12-new} and doing some more work 
one can eventually rule out weak convergence. We now give the rigorous proof of Counterexample~\ref{l:ce1}. 
\begin{proof}[Proof of Counterexample~\ref{l:ce1}] 
We proceed according to the following steps. \\
{\sc Step 1:} we investigate the structure of the entropy solution $u$. 
First, we collect some properties of $u$:
\begin{enumerate}
\item[a)] $u \in C^0 ([0, + \infty[; L^1 (\R))$.
\item[b)] Since $\| \bar u \|_{L^\infty} \leq 1$, then by the maximum principle $\| u (t, \cdot) \|_{L^\infty} \leq 1$ for every $t \ge 0$.
\item[c)] Since $\bar u \in BV (\R)$, then $u (t, \cdot) \in BV(\R)$ for every $t \ge 0$. 
\item[d)] A  $0$-speed shock is created at $t=0$ at the origin $x=0$. Owing to the finite propagation speed, this 
shock will survive for some time. More precisely,
 we have 
\begin{equation}
\label{e:traccia}
  u(t, x) = 
  \left\{
  \begin{array}{ll}
  1 & \text{for a.e $x \in ]-1/2, 0[$} \\
  -1 & \text{for a.e $x \in ]0, 1/2[$\,,} \\ 
    \end{array}
  \right. \quad 
\text{for every $t \in [0, 1/4]$}. 
\end{equation} 
\item[e)] Owing to the finite propagation speed and to the fact that $\bar u = 0$ if $|x|>2$, we have $u(t, x)=0$ for a.e. $|x| \ge 3$ and for every $t \in [0, 1/4]$.  
\end{enumerate}
We now want to show that 
\begin{equation}
\label{e:noconv} 
    \int_{-4}^0 u (1/4, x)  dx = 
    \int_{-4}^0 \bar u (x) dx - \frac{1}{4}.
\end{equation}
We can \emph{formally} obtain~\eqref{e:noconv} by pointing out that 
$$
    \frac{d}{dt} \int_{-4}^{0} u (t, x) dx \stackrel{\eqref{e:cpburgers}}{=}
     - \int_{-4}^{0} \partial_{x} [u^2] (t, x)dx =
          - u^2 (t, 0^-) + u^2 (t, -4) \stackrel{\mathrm{d), \ e)}}{=} -1
$$
 and by integrating with respect to time.  
We now sketch a rigorous argument to justify~\eqref{e:noconv}. First, we point out that $u$ is a distributional solution of~\eqref{e:cpburgers}, which amounts to say that 
\begin{equation}
\label{e:distrsol}
\int_0^{+ \infty}  \! \! \! \int_\R   u \partial_t \varphi \, dx dt + 
\int_0^{+ \infty} \! \! \!  \int_\R u^2 \partial_x  \varphi \, dx dt
+ \int_\R \varphi(0, \cdot ) \bar u \, dx =0
\end{equation}
for every $\varphi \in C^\infty_c (\R^2)$. We now introduce the sequence of functions $\{ \chi_n \} \subseteq C^\infty_c (\R)$ such that
\begin{equation}
\label{e:chienne}
       \chi_n (x) = 
       \left\{ 
       \begin{array}{ll}
       1 & -4 + 1/n \leq x \leq -1/n \\
       0 & x \leq -4 \; \text{or} \; x \ge 0. \\
       \end{array}
      \right. 
     %\quad 0 \leq \chi'_n (x) \; \text{if $x \leq - 4+1/n$}, \; \chi'_n (x) \leq 0 
%        \; \text{if $x \ge -1/n$}
\end{equation}
As a matter of fact, $\chi_n$ is an approximation of the characteristic function of 
$[-4, 0]$. We fix an arbitrary $\theta \in C^\infty_c (]0, 1/4[)$, we plug $\varphi_n(t, x) : = \chi_n (x) \theta (t)$ as a test function in~\eqref{e:distrsol} and we point out that 
\begin{equation*}
\begin{split}
    \int_0^{+ \infty} \! \! \!  \int_\R u^2 \partial_x  \varphi_n \, dx dt & =
    \int_0^{1/4} \! \theta (t)  
    \int_\R u^2 (t, x) \chi_n'(x) dx dt \\ &
    \stackrel{\eqref{e:chienne}}{=}
    \int_0^{1/4} \! \theta (t)  
    \int_{-4}^{-4 + 1/n} \! \! 
    u^2 (t, x) \chi_n'(x) dx dt +
    \int_0^{1/4} \! \theta (t)  
    \int_{-1/n}^{0} \! \! 
    u^2 (t, x) \chi_n'(x) dx dt   \\ &
    \stackrel{\eqref{e:traccia},\, \mathrm{e)}}{=}
     \int_0^{1/4} \! \theta (t)  
    \int_{-4}^{-4 + 1/n} \! \! 
    0 \cdot  \chi_n'(x) dx dt +
    \int_0^{1/4} \! \theta (t)  
    \int_{-1/n}^{0} \! \! 
    1 \cdot\chi_n'(x) dx dt \\ & 
     \stackrel{\eqref{e:chienne}}{=}
     \int_0^{1/4} \! \theta (t) (-1) dt . 
\end{split}
\end{equation*}
Next, we let $n \to + \infty$ in the other term in~\eqref{e:distrsol} and we 
eventually arrive at 
$$
  \int_0^{1/4} \! \!  \theta' (t)  \int_{-4}^0   u (t, x) dx  \, dt + \int_0^{1/4} \theta(t)
  (-1) dt = 0.  
$$
Owing to the arbitrariness of $\theta$, this implies that the continuous function 
\begin{equation}
\label{e:continuita}
   t \mapsto \int_{-4}^0   u (t, x) dx 
\end{equation}
has distributional derivative equal to $-1$. This implies that the above function is actually absolutely continuous and, owing to the Fundamental Theorem of Calculus, we get~\eqref{e:noconv}.  

Since we will need it in the following, we also point that, since the map in~\eqref{e:continuita} is continuous, then~\eqref{e:noconv} implies that there is 
$h>0$ such that 
$$
    \int_{1/4 - h}^{1/4+ h} \int_{-4}^0   u (t, x) dx dt \leq 
    \int_{1/4 - h}^{1/4+ h} \left( \int_{-4}^0   \bar u \ dx - \frac{1}{8} \right) dt =
    2h \int_{-4}^0  \bar u \ dx - \frac{h}{4}. 
$$
In other words, if we define $E$ by setting 
\begin{equation}
\label{e:insieme}
    E : = \big\{ (t, x): \, t \in [ 1/4-h, 1/4 +h], \; x \in [-4, 0]  \big\}
\end{equation}
and we denote by $\mathbf{1}_E$ the characteristic function of $E$, then 
\begin{equation}
\label{e:nebbia}
       \int_0^{\infty} \! \! \int_\R \mathbf{1}_E \,u \, dx dt \leq 
       2h \int_{-4}^0   \bar u \ dx - \frac{h}{4}. 
\end{equation}
{\sc Step 2:} we show that the distributional solution $u_\ee$ of~\eqref{e:cpce} is odd, namely that, for a.e. $(t, x) \in \R^+ \times \R$, $u(t, x) = - u(t, -x)$. 
We set 
$
    v_\eps(t, x) : = - u_\eps (t, -x).
$ 
If we can prove that $v_\ee$ is also a distributional solution of the Cauchy problem~\eqref{e:cpce}, then by the uniqueness part of Proposition~\ref{p:exunice} we get that for every $t \ge 0$ it holds $v_\ee (t, x)= u_\ee(t, x)$ for a.e. $x$, namely that $u_\ee$ is an odd function.

To show that $v_\ee$ is a distributional solution of~\eqref{e:cpce}, we first observe that by using the fact that 
$\eta_\eps$ is even and making the change of variables $z = - y$ we get 
\begin{equation}
\label{e:oddcon}
    \big(v_\eps \ast \eta_\eps \big) (t, x)   = 
    - \int_\R u_\eps (t, -x + y ) \eta_\eps (y) dy = 
    - \int_\R u_\eps (t, -x - z) \eta_\eps (z) dz  =
    - \big( u_\eps \ast \eta_\eps \big) (t, - x). 
\end{equation}
Next, we fix $\varphi \in C^\infty_c ( \R^2)$, we set $\phi (t, x) : = - \varphi(t, -x)$ and we obtain 
\begin{equation*}
\begin{split} 
   & \int_0^{+ \infty}  \! \! \! \int_\R   v_\eps \partial_t \varphi \, dx dt + 
\int_0^{+ \infty} \! \! \!  \int_\R v_\eps (v_\eps \ast \eta_\eps) \partial_x  \varphi \, dx dt
+ \int_\R \varphi(0, \cdot ) \bar u \, dx = [ z = -x  ] \\
& = 
 \int_0^{+ \infty}  \! \! \! \int_\R (- u_\eps)  (- \partial_t \phi ) dz dt + 
 \int_0^{+ \infty} \! \! \!  \int_\R (-u _\eps) (- u_\eps \ast \eta_\eps) ( \partial_x  \phi)  \, 
 dz dt 
 +  \int_\R (- \phi(0, \cdot ))(- \bar u ) dz %\\ &
  =0. 
\end{split}
\end{equation*}
To establish the last equality we have used the fact that $u_\ee$ is a distributional solution of~\eqref{e:cpce}. The above chain of equalities states that $v_\ee$ is a distributional solution of~\eqref{e:cpce} and hence concludes {\sc Step~2}. \\
{\sc Step 3:} we show that 
\begin{equation}
\label{e:ancheufazero}
       u_\ee (t, x) =0, \quad \text{for a.e. $|x|\ge 2$ and every $t\ge0$ and $\ee >0$}. 
\end{equation}
We note that 
 $u_\ee$ is a distributional solution of the Cauchy problem 
\begin{equation}
\label{e:kruzkov}
\left\{
\begin{array}{ll}
    \partial_t u_\ee + \partial_x \big[ u_\ee g_\ee \big] =0 \\
    u_\ee(0, \cdot) = \bar u  \\
\end{array}
\right.
\end{equation}
provided that the vector field $g_\ee$ is defined as $g_\ee (t, x) : = u_\ee \ast \eta_\ee$. Since the vector field $g_\ee$ is smooth, then we can apply the method of characteristics. We term $X(t, x)$ the characteristic curve solving the Cauchy problem 
\begin{equation}
\label{e:caratteristiche}
  \left\{
  \begin{array}{ll}
  \displaystyle{\frac{dX}{dt} }= g_\ee (t, X) \\
  \\
  X(0, x) = x . \\
  \end{array}
  \right.
\end{equation}
Recall that $u_\ee$ is an odd function by {\sc Step 2}. Since $\eta_\ee$ is an even function by assumption, by arguing as in the chain of equalities~\eqref{e:oddcon} we obtain that $u_\ee \ast \eta_\ee$ is an odd function. Since it is also smooth, we eventually conclude that  
\begin{equation}
\label{e:fazero}
  g_\ee(t, 0) = u_\ee \ast \eta_\ee(t, 0) = 0, \quad \text{for every $t \ge 0$.}
\end{equation}
This means that $X(t, 0) \equiv 0$ and, since~\eqref{e:caratteristiche} has a unique solution, implies that the characteristic curves \emph{cannot cross} the $t$ axis. Since $\bar u(x) \ge 0$ if $x \leq 0$ and $\bar u(x) \leq 0$ if $x \ge 0$, this in turn implies that 
\begin{equation}
\label{e:segno}
     \text{for every $t \ge 0$}, \; u_\ee (t, x) \ge 0 \; \text{for a.e. $x <0$ and} \;  
     u_\ee (t, x) \leq 0 \; \text{for a.e. $x >0$.}
\end{equation}
This implies that, if $x \leq -2$, then $x + \ee \leq 0$ and hence 
\begin{equation}
\label{e:versodestra}
         g_\ee (t, x) = u_\ee \ast \eta_\ee (t, x) = 
         \int_{x - \ee}^{x+ \ee} u_\ee (y) \eta_\ee (x-y) dy 
         \stackrel{\eqref{e:v},\eqref{e:segno}}{\ge}0 .
\end{equation}  
If $x_1 < x_2$, then $X(t, x_1) < X(t, x_2)$ for every $t \ge 0$ (to see this, we use again the fact that the solution of~\eqref{e:caratteristiche} is unique). By recalling~\eqref{e:versodestra}, this implies that 
$$
    x \ge - 2 \implies X(t, x) \ge - 2 \; \text{for every $t \ge 0$}
$$
and hence that 
$$
    \text{for every $t \ge 0$,} \ \; X(t, x) <  -2 \implies  x < -2  . 
$$
Since $\bar u(x) =0$ for a.e. $x \leq 2$, this eventually implies that $u_\ee(t, x) =0$ for every $x \leq - 2$. Since the function $u_\ee$ is odd, this establishes~\eqref{e:ancheufazero}. \\
{\sc Step 4:} we conclude the proof. Recall that the set $E$ is defined as in~\eqref{e:insieme} for a suitable $h$ and assume that we have shown that 
\begin{equation}
\label{e:piove}
    \int_0^{+\infty} \! \! \! \int_\R \mathbf{1}_E u_\ee  dx dt = 
     2 h \int_{-4}^0  \bar u \ dx. 
\end{equation}
Since the function $\mathbf{1}_E \in L^p (\R^+ \times \R)$, for every $p \in [1, + \infty]$, then by comparing~\eqref{e:piove} and~\eqref{e:nebbia} we rule out the possibility that $u_\ee$ converges weakly or weakly$^\ast$ to $u$.  
We are thus left with establishing~\eqref{e:piove}. To this end, we first use the formal computation 
\begin{equation}
\label{e:formal}
   \frac{d}{dt} \int_{-4}^0 u_\ee (t, x) dx \stackrel{\eqref{e:cpce}}{=} 
   - \int_{-4}^0 \partial_x \big[ u_\ee (u_\ee \ast \eta_\ee) \big] 
   (t, x) dx = u_\ee (u_\ee \ast \eta_\ee) (t, -4)  - u_\ee (u_\ee \ast \eta_\ee) (t, 0) 
   \stackrel{\eqref{e:ancheufazero},\eqref{e:fazero}}{=} 0.
\end{equation}
This implies that 
$$
    \int_0^{+\infty} \! \! \! \int_\R \mathbf{1}_E u_\ee  dx dt=
    \int_{1/4-h}^{1/4+h}  \! \! \! \int_{-4}^0
     u_\ee  dx dt = \int_{1/4-h}^{1/4+h}
     \! \! \! \int_{-4}^0
     \bar u(x)  dx dt = 2 h \int_{-4}^0
     \bar u (x)  dx,
$$
namely~\eqref{e:piove}. 
 To provide a rigorous justification of~\eqref{e:formal} one can argue as in {\sc Step 1}. This concludes the proof of the lemma.  
\end{proof}
\subsection{A counterexample with positive data and asymmetric kernels}
\label{ss:ce2}
 This paragraph aims at establishing the following lemma, which rules  out also the possibility that $u_{\ee_k}$ weakly converges to a distributional, not necessarily entropy admissible, solution of~\eqref{e:cpburgers}. 
\begin{counterexample}
\label{l:ce2}
Assume that $\eta_\ee$ satisfies~\eqref{e:eta} and~\eqref{e:etaee} and moreover that 
\begin{equation}
\label{e:asymmetric}
  \eta(x) = 0, \quad \text{for every $x \ge 0$.}
\end{equation}
Let $\bar u$ be given by 
\begin{equation}
\label{e:u0ce3}
\bar u(x) =
\left\{
\begin{array}{ll}
1 & -1<x < 0 \\
0 & \text{otherwise.} \\
\end{array}
\right.
\end{equation}
Let $u_\ee$ be the solution of the Cauchy problem~\eqref{e:cpce}, \eqref{e:u0ce3} and $u$ be the entropy admissible solution of~\eqref{e:cpburgers}, \eqref{e:u0ce3}. Then
\begin{enumerate}
\item the family of distributional solutions $\{ u_\ee \}_{\ee >0}$ does not converge to $u$, not even in the weak topology of $L^p$, $p \geq 1$, in the weak$^*$ topology of $L^\infty$, or up to subsequences;
\item more in general, any weak limit $w$ of a subsequence of $\{ u_\ee \}_{\ee >0}$ (in the weak topology of $L^p$, $p \geq 1$, in the weak$^*$ topology of $L^\infty$) cannot be a $L^2_{\loc}$ distributional (not necessarily entropy admissible) solution of~\eqref{e:cpburgers}.
\end{enumerate}
\end{counterexample}

The basic idea underpinning Counterexample~\ref{l:ce2} %and~\ref{l:ennio2} 
 is, very loosely speaking,  the following. Owing to~\eqref{e:asymmetric}, 
the convolution $u_\ee \ast \eta_\ee$ evaluated at the point $x$ only depends on the values of $u_\ee$ \emph{on the right hand side} of $x$. Owing to the particular structure of the initial datum $\bar u$ this implies that $u_\ee \ast \eta_\ee(0, 0)=0$ and hence that the characteristic line of the velocity field  $u_\ee \ast \eta_\ee$ starting at $x=0$ has zero initial speed. Then, one can show that the speed is identically zero:
this  implies that the characteristic lines coming from the half line $x<0$ cannot cross the axis $x=0$, and hence that no mass can enter the half line $x>0$. 
In conclusion, $u_\ee (t, x) = 0$ for a.e. $x>0$. Notice that this last equality could be shown also by noticing that the approximating sequence in the construction of $u_\ee$ in \cite[\S~5]{CLM} enjoys the same property.

On the other hand, the entropy admissible solution of~\eqref{e:cpburgers}, \eqref{e:u0ce3} is explicit and not identically $0$ for $x>0$. With some more work, one can show that any \emph{distributional} solution of~\eqref{e:cpburgers},~\eqref{e:u0ce3} is not identically $0$ for $x>0$. This allows to rule out weak convergence to a distributional solution. We now make the previous argument rigorous.% by first of all establishing the following result. 
\begin{lemma}
\label{l:suppuee}
Assume that $\eta$ and $\eta_\ee$ satisfy~\eqref{e:eta},~\eqref{e:etaee} 
and~\eqref{e:asymmetric} and let $\bar u$ be as in~\eqref{e:u0ce3}. Then 
\begin{equation}
\label{e:suppuee}
\text{for every $t \ge 0$, $u_\ee (t, x) =0$ for a.e. $x <-1$ and a.e. $x >0$.
}
\end{equation}
\end{lemma}
\begin{proof}
We argue according to the following steps. \\
{\sc Step 1:} we show that $u_\ee (t, x) =0$ for a.e. $x <-1$. 
We use the method of characteristics: note that $u_\ee$ is a distributional solution of the continuity equation~\eqref{e:kruzkov} provided the vector field $g_\ee$ is given by $g_\ee: = u_\ee \ast \eta_\ee$. Since $\bar u\ge0$, then $g_\ee(t, x) \ge 0$ for every $(t, x)$. This implies that, for every $t \ge 0$ and every $x < -1$, the characteristic line $Y_t(s, x)$ solving the (backward) Cauchy problem 
\begin{equation}
\label{e:backward}
    \left\{
    \begin{array}{ll}
    \displaystyle{\frac{d Y_t}{ds}} = g_\ee (s, Y_t) \\
    \phantom{ciao} \\
    Y_t(t, x)=x \\
    \end{array}
     \right.
\end{equation}
satisfies $Y_t(0,x) < -1$ and hence $\bar u (Y_t(0,x)) =0$. Since the value $0$ is propagated along the characteristic lines of the continuity equation, then $u_\ee(t, x)=0$.  \\
{\sc Step 2:} we regard again $u_\ee$ as the solution of the continuity equation~\eqref{e:kruzkov} and  
we term $X$ the characteristic line solving the (forward) Cauchy problem~\eqref{e:caratteristiche}. We claim that
\begin{equation}
\label{e:icstzero-new}
X(t, x) =x \quad \text{for every $t \ge 0$, $x\ge 0$}.
\end{equation}
Indeed, by the spatial smoothness of the vector field $u_\ee \ast \eta_\ee$, 
the characteristic lines ``cannot cross" the curve $X(t, 0)$; in particular for any $t>0$ and $x>X(t, 0)$ we have $Y_t(0,x)>0$. Hence
\begin{equation}
\label{eqn:u-ee0}
u_\ee(t, x)=0 \qquad \mbox{for any } t>0, x\geq X(t,0).
\end{equation}
Since $\eta_\ee$ satisfies~\eqref{e:eta} and~\eqref{e:etaee}, for any $x\in \R$ the quantity $u_\ee \ast \eta_\ee(x)$ is an average, weighted with $\eta_\ee$, of the values of $u_\ee(t,\cdot)$ on the right of $x$. From \eqref{eqn:u-ee0}, we deduce that
\begin{equation}
\label{eqn:u-ee0eta}
u_\ee \ast \eta_\ee(t,  x)=0 \qquad \mbox{for any } t>0, x \geq X(t,0).
\end{equation}
From \eqref{eqn:u-ee0eta} applied to $x= X(t,0)$ and \eqref{e:caratteristiche} with $x=0$, we deduce that $X(t,0)= 0$ for any $t>0$; applying again \eqref{eqn:u-ee0eta}  with this further information, we deduce \eqref{e:icstzero-new}.

Since the value $0$ is propagated along the characteristic lines of the continuity equation, which in turn are constant for any $x\geq 0$ thanks to  \eqref{e:icstzero-new}, we have shown that $u_\ee(t, x)=0$ for any $x\geq 0$, concluding the proof of \eqref{e:suppuee}.
\end{proof}
%We can now provide the 
\begin{proof}[Proof of Counterexample~\ref{l:ce2}(1)]
First, we point out that, if $\bar u$ is given by~\eqref{e:u0ce3}, then the entropy admissible solution of the Cauchy problem~\eqref{e:cpburgers} is 
\begin{equation}
\label{e:limitece3}
u(t, x)
= 
\left\{
\begin{array}{ll}
0 & x \leq -1 \mbox{ or }x \ge t \\
\displaystyle{\frac{x+1}{2t}} & - 1 \leq x \leq 2t-1 \\
1 & 2t -1 \leq x \leq t , \\
%0 & x \ge t, \\
\end{array}
\right. \quad \text{for a.e. $(t, x) \in [0,1 ] \times \R$.}
\end{equation}
Assume by contradiction there is a sequence $\{ \ee_k \}$ such that $u_{\ee_k}$ weakly converges to $u$. 
We use as a test function the characteristic function of the set $E:=[0, 1/2] \times [0, 1]$. Since 
$$
   \int_{\R^+ \times \R} u_{\ee_k} \mathbf{1}_E \ dx dt 
   \stackrel{\text{Lemma~\ref{l:suppuee}}}{=}0, 
   \qquad 
   \int_{\R^+ \times \R} u \ \mathbf{1}_E \ dx dt
   \stackrel{\eqref{e:limitece3}}{=}
   \int_0^{1/2} \! \! \int_0^t 1 \ dx  dt = \frac{1}{8}, 
$$
then we find a contradiction. 
\end{proof}
The proof of Counterexample~\ref{l:ce2}(2) is based on the following result, which could be generalized to Young measure solutions of the Cauchy problem~\eqref{e:cpburgers} (we refer to~\cite{Dafermos:book} for an extended discussion on Young measures and their applications to nonlinear conservation laws). 
\begin{lemma}
\label{l:ennio}
Let $a,b\in \R$, $a<b$, and let $u \in L^2_{\mathrm{loc}} ([0, 1]
 \times \R
)$ be a nonnegative, distributional solution of the Cauchy problem~\eqref{e:cpburgers} compactly supported in $[0,1]\times (a,b)$. %and 
%\begin{equation}
%\label{e:431}
%    u (t, x) \ge 0 \; \;  \text{for a.e. $(t, x) \in [0,1] \times \R$}. 
%\end{equation}
Then the baricenter of $u$ is a nondecreasing function and 
\begin{equation}
\label{e:nono}
   \int_{a}^{b} \! 
  x u(t, x) dx \ge \Big(\int_{a}^b \bar u\Big)^2 t +   \int_{a}^{b} \! \! \!
  x \bar u(x) dx.
\end{equation}
\end{lemma}
The proof of Counterexample~\ref{l:ce2}(2) straightforwardly follows from Lemma~\ref{l:ennio}. Indeed, 
any nonnegative distributional solution $u$ of the Cauchy problem~\eqref{e:cpburgers} starting from $\bar u$ in \eqref{e:u0ce3} cannot satisfy
\begin{equation}
\label{e:430new}
      u(t, x)=0 \; \; \text{for a.e. $t\in [0, 1] ,\;  x\in  (-\infty, -1) \cup (0, \infty)$},
\end{equation}
because otherwise it would contradict \eqref{e:nono} for $a=-1-\sigma,b=\sigma$ ($\sigma$ arbitrarily small) and any $t>1/2$. 
Hence we find a contradiction with \eqref{e:suppuee} as in the proof of Counterexample~\ref{l:ce2}(1). 
\begin{proof}[Proof of Lemma~\ref{l:ennio}]
%We argue by contradiction: we assume that there is a distributional solution of the Cauchy problem~\eqref{e:cpburgers}.  By contradiction, let $u \in L^2_{\mathrm{loc}} ([0, 1] \times \R)$ be a nonnegative distributional solution of the Cauchy problem~\eqref{e:cpburgers} satisfying \eqref{e:430}.
The conservation law \eqref{e:cpburgers} implies
\begin{equation}
\label{e:conserviamo}
    \int_{a}^{b} u(t, x) dx =
    \int_\R u(t, x) dx 
    =
    \int_{a}^{b} \bar u(x) dx 
 \quad \text{for a.e. $t>0$}. 
\end{equation}
We perform some \emph{formal} computations, which can be made rigorous by arguing as in the proof of Counterexample~\ref{l:ce1}: by \eqref{e:cpburgers}, the Jensen inequality, and the previous equality, we have
\begin{equation}
\label{e:baricentro}
\begin{split}
\frac{d}{dt}
 \int_{a}^{b}& 
  x \, u(t, x) dx  =
 \int_{a}^{b} x \, \partial_t u (t, x) dx 
=
 - \int_{a}^{b} 
x \, \partial_x \big[ u^2 \big](t, x) dx
 \\ & 
 = \Big[ - x \, u^2 (t, x)\Big]^{x = a}_{x = b}
 + 
 \int_{a}^{b} u^2 (t, x)dx 
\ge 
  \frac{1
 }{b-a}
  \left( \int_{a}^{b}  u (t, x)dx \right)^2.
\end{split}
\end{equation} 
Integrating in time, we get \eqref{e:nono}.
%On the other hand, since $u \ge 0$, 
%In particular, the left hand side in \eqref{e:nono} attains positive values when
%$$
%   \int_{-1}^{0} \! 
%   x u(t, x) dx  \leq 0,
%$$
%which contradicts~\eqref{e:nono} 
%$t > 1/2$, which contradicts the fact that $u$ is nonnegative. 
\end{proof}

%\begin{remark}
%\label{r:young}
%As a matter of fact, one can show that Lemma~\ref{l:ennio} can be extended to the case of Young measures that solve 
%%the equation.
%%
%%to rule out the possibility that a family of Young measures $\big\{ \nu_{(t, x)} \big\}$ {\color{red} satisfying~\eqref{e:430new} and supported on the nonnegative axis} is a distributional solution of 
%the Cauchy problem~\eqref{e:cpburgers}. 
%We refer to~\cite{Dafermos:book} for an extended discussion on Young measures and their applications to nonlinear conservation laws. 
%\end{remark}
%%{\color{blue}Nota di Laura: se volete possiamo enunciare il Lemma~\ref{l:ennio} nel caso delle misure di Young. A me sembra che si appesentirebbe molto la notazione, senza aggiungere moltissimo all'enunciato, per questo ho preferito mettere solo un remark. Ditemi per\`o voi}.
\subsection{A counterexample with positive data and symmetric kernels}
\label{ss:ce3}
We now establish the following result. 
\begin{counterexample}
\label{l:ce3}
Assume that $\eta$ and $\eta_\ee$ are as in~\eqref{e:v} and~\eqref{e:etaee}, respectively, and that $\eta$ is an even function. Let $u$ denote the entropy admissible solution of~\eqref{e:cpburgers},~\eqref{e:u0ce3} and $u_\ee$ the solution of~\eqref{e:cpce},~\eqref{e:u0ce3}. 
Then for every $\delta>0$, the family $u_\ee$ does not converge to $u$ strongly in $L^{1+\delta} $, not even up to subsequences. More precisely, 
\begin{equation}
\label{e:ce3}
        \forall \, t  >0, \, \nexists \; \text{$\{ \ee_k \}$,} \; \ee_k \to 0^+  \; \text{such that} \; u_{\ee_k} (t, \cdot) \to u(t, \cdot) \; 
        \text{strongly in $L^{1 + \delta}$.}         
\end{equation}
\end{counterexample}
Note that~\eqref{e:ce3} rules out the possibility that $u_\ee$ converges to $u$ in $L^{1+ \delta} ([0, 1] \times \R)$: indeed, if this were true then, up to subsequences, $u_\ee(t, \cdot) \to u (t, \cdot)$ in $L^{1+ \delta} $ for a.e. $t$, and this is ruled out by~\eqref{e:ce3}. 

The basic idea underpinning Counterexample~\ref{l:ce3} is the following. We introduce the entropy function 
$$
    \mathscr E (u): = \int_\R u \ln u \ dx, 
$$
where by a slight abuse of notation we have continuously extended the function $u \ln u$ with value $0$ for $u=0$.
%following function, defined for $u\ge 0$: 
%$$
%   u \mapsto 
%   \left\{
%   \begin{array}{ll}
%   u \ln u & u > 0 \\
%   0 & u =0. \\
%   \end{array}
%   \right.
%$$
By using the formal computation~\eqref{e:savare}, one gets that 
$$
     \frac{d}{dt} \mathscr E(u_\ee) = 0
$$
if $u_\ee$ is a nonnegative solution of~\eqref{e:cpce}. On the other hand, the function $u \ln u$ is convex and hence $\mathscr E(u)$ is non increasing for nonnegative entropy admissible solutions of~\eqref{e:cpburgers}. In particular, if the initial datum is as in~\eqref{e:u0ce3}, then $\mathscr E(u)$ is strictly decreasing. After some more work this allows us to rule out the strong convergence of $u_\ee$ to $u$. 
 
The precise argument requires some preliminary results.
\begin{lemma}
\label{l:continuity}
Fix $\delta >0$ and assume that  $\{ v_k \} \subseteq L^{1+ \delta} $ satisfies 
   $v_k \to v$ in $L^{1+\delta} $, for some  compactly supported $v \in L^{1+\delta}$. Then 
\begin{equation}
\label{e:tesilemma-new}
   \int_\R v \ln v  \, dx \ge     \limsup_{k \to + \infty} 
    \int_\R v_k \ln v_k  \, dx. 
\end{equation}

   \end{lemma}
\begin{proof}
Let $\Omega \subset \R$ be a compact set s.t. $v=0$ a.e. in $\R \setminus \Omega$. Up to a (not relabelled) subsequence, we can assume that the $\limsup$ in the right-hand side in \eqref{e:tesilemma-new} is a limit.  Since $v_k \to v$ in $L^{1+\delta} $, up to a further subsequence, we can assume that $v_k$ converges pointwise to $v$ a.e. in $\R$ and that there exists a function $h \in L^{1+\delta}(\R)$ such that $h \geq |v_k|$ a.e. for any $k \in \mathbb N$. We observe that
\begin{equation}
\label{e:boundsuvk}
       | w \ln w  | \leq C(\delta) \big(\mathbf{1}_{\{w< 1\}} + \mathbf{1}_{\{w\geq 1\}}  |w|^{1+\delta}) , 
       \quad \text{for every $w\ge 0$}.
\end{equation}
Since the function $s \to s\ln s$ is negative for $s<1$, by \eqref{e:boundsuvk}, and since $v_k \to v$ in $L^{1+\delta} (\R\setminus \Omega)$ we have
$$\limsup_{k\to \infty} \int_{\R \setminus \Omega} v_k\ln v_k \, dx \leq \limsup_{k\to \infty} \int_{\R \setminus \Omega} v_k\ln v_k \mathbf{1}_{v_k\geq 1} \, dx \leq C(\delta) \lim_{k\to \infty} \int_{\R \setminus \Omega} |v_k|^{1+\delta}\, dx =0 .
$$
Since the functions $v_k \ln v_k$ converge pointwise to $v \ln v$ as $k\to \infty$ and the convergence is dominated by $C(\delta) (1+ |h|^{1+\delta}) \in L^1(\Omega)$, we have
$$\int_\Omega v\ln v \, dx= \lim_{k\to \infty} \int_\Omega v_k\ln v_k \, dx \geq  \lim_{k\to \infty} \int_\Omega v_k\ln v_k \, dx + \limsup_{k\to \infty} \int_{\R \setminus\Omega} v_k\ln v_k \, dx =  \limsup_{k\to \infty} \int_{\R } v_k\ln v_k \, dx,
$$
which proves \eqref{e:tesilemma-new}.
\end{proof}
%The second preliminary result is the following. 
\begin{lemma}
\label{l:entropyconservation}
Let $u_\ee$ be the solution of~\eqref{e:cpburgers} and assume that $\eta_\ee$ satisfies~\eqref{e:v} and~\eqref{e:etaee} and that $\eta$ is an even function. Let $\bar u \in L^\infty$ be a compactly supported function satisfying $\bar u \ge 0$ and  
\begin{equation}
\label{e:ipo1}
    \int_\R \bar u \ln \bar u \ dx  < + \infty .
\end{equation}
Then for every $t\ge 0$ the distributional solution satisfies the following properties: $u_\ee(t, \cdot) \ge0$ and 
\begin{equation}
\label{e:entropyconservation}
    \int_\R u_\ee \ln u_\ee (t, \cdot) \ dx = \int_\R \bar u \ln \bar u \ dx.  
\end{equation}
\end{lemma}
\begin{proof}
First, we point out that one can check by direct computation that, for every 
$a, b \in L^2 (\R)$,  $c \in C^\infty_c (\R)$, we have 
\begin{equation}
\label{e:abc}
    \int_\R a (b \ast c) dx = \int_\R (a \ast \check c) b \, dx,   
\end{equation}
where $\check c(x)= c(-x).$ We apply the above formula with $a= b=u_\ee (t, \cdot)$ and  $c= \eta'_\ee$. Since $\eta_\ee$ is an even function, then the derivative $\eta_\ee'$ is an odd function and hence $\check \eta'_\ee = - \eta_\ee'$. We then obtain 
$$
    \int_\R u_\ee \partial_x \big[ u_\ee \ast \eta_\ee \big] (t, \cdot)dx = 
    \int_\R u_\ee  ( u_\ee \ast \eta'_\ee ) (t, \cdot)dx \stackrel{\eqref{e:abc}}{=}
    - \int_\R (u_\ee \ast \eta_\ee') u_\ee (t, \cdot) dx=
    - \int_\R u_\ee \partial_x \big[ u_\ee \ast \eta_\ee \big] (t, \cdot)dx ,   
$$
which implies that 
\begin{equation}
\label{e:abc2}
     \int_\R u_\ee \partial_x \big[ u_\ee \ast \eta_\ee \big] (t, \cdot)dx  = 0. 
\end{equation}
We can then establish~\eqref{e:entropyconservation} by using the
following 
(formal) computation: 
\begin{equation}
\label{e:savare}
\begin{split}
  \frac{d}{dt} & \int_\R u_\ee \ln u_\ee (t, \cdot) dx
  =  \int_\R
   (1 + \ln u_\ee) \partial_t u_\ee (t, \cdot) dx \stackrel{\eqref{e:cpburgers}}{=}
    - \int_\R (1 + \ln u_\ee) \partial_x \big[ u_\ee (u_\ee \ast \eta_\ee) \big] (t, \cdot) 
    dx \\
    & =    
   \int_\R   \partial_x u_\ee \, \frac{1}{u_\ee} \,  u_\ee (u_\ee \ast \eta_\ee) (t, \cdot) dx 
    = 
    \int_\R  \partial_x u_\ee (u_\ee \ast \eta_\ee) (t, \cdot) dx
    = - \int_\R u_\ee \partial_x 
    \big[  u_\ee \ast \eta_\ee\big] (t, \cdot) dx
    \stackrel{\eqref{e:abc2}}{=}0.
\end{split}
\end{equation}
To make the above argument rigorous, we recall that $u_\ee$ is the solution of the Cauchy problem~\eqref{e:kruzkov}, where the velocity field $g_\ee = u_\ee \ast \eta_\ee$ is smooth. By the renormalization property, for every $\beta \in C^1$ we have~\eqref{e:renormalized}. This implies that 
\begin{equation}
\label{e:rino2}
    \int_\R \beta (u_\ee (t, \cdot)) dx  = 
    \int_\R \beta (\bar u) dx
    - \int_0^t \int_\R \partial_x
    \big[ u_\ee \ast \eta_\ee \big] 
    \Big[ 
    u_\ee \beta'(u_\ee) - \beta (u_\ee) \Big] dx ds. 
\end{equation}
 We construct a sequence of functions $\beta_n: \R^+ \to \R$ by setting 
 $$
     \beta_n (v): = 
     \int_0^v \left[ 1 + \ln \left( \xi + 
     \frac{1}{n} 
     \right) \right] d \xi.  
 $$
 Note that 
 $$
     \beta_n (v) \to 
     \int_0^v \left[ 1 + \ln \xi \right] d \xi = 
     v \ln v, \quad 
     \text{for every $v \ge 0$, as $n \to + \infty$},
 $$
 $$
     v \beta'_n (v) - \beta_n (v) \to 
     v, \quad 
      \text{for every $v \ge 0$, as $n \to + \infty$}. 
 $$
 By testing the inequality~\eqref{e:rino2} with $\beta_n$ and passing to the limit for $n \to + \infty$ we obtain 
 \begin{equation}
\label{e:rino3}
    \int_\R u_\ee \ln u_\ee \ dx  = 
    \int_\R \bar u \ln \bar u \ dx
    - \int_0^t \int_\R \partial_x
    \big[ u_\ee \ast \eta_\ee \big] 
    u_\ee \ ds dx
    \stackrel{\eqref{e:abc2}}{=}
    \int_\R \bar u \ln \bar u \ dx. 
\end{equation}
This concludes the proof of the lemma. 
\end{proof}
%We can now complete the proof of Lemma~\ref{l:ce3}. 
\begin{proof}[Proof of Counterexample~\ref{l:ce3}]
If $\bar u$ is given by~\eqref{e:u0ce3}, then the entropy admissible solution $u$ can be explicitly computed and is given by~\eqref{e:limitece3}. 
Note that $\bar u$ only attains the values $0$
 and $1$, whereas  if $t \in (0, 1]$ then $u$ attains values between $0$ and $1$. Therefore 
\begin{equation}
\label{e:unozero}
  \int_\R u(t,x) \ln u  (t, x) \ dx< 0 =    \int_\R \bar u \ln \bar u \ dx  \qquad\mbox{for any }t \in (0, 1]. 
\end{equation} 
Owing to Lemma~\ref{l:entropyconservation}, for every $\ee >0$ and $t \ge 0$ we have 
\begin{equation}
\label{e:sizero}
     \int_\R u_\ee(t,x) \ln u_\ee (t,x) \ dx = 0. 
\end{equation}
Assume by contradiction  that there is a sequence $\ee_k \to 0^+$ and a time $t >0$ such that $u_{\ee_k} (t, \cdot) \to u(t, \cdot)$ strongly in $L^{1 +\delta} (\R).$  We apply 
Lemma~\ref{l:continuity} with $v_k : = u_{\ee_k} (t, \cdot)$, $v: = u(t, \cdot)$. By combining~\eqref{e:sizero} and~\eqref{e:tesilemma-new} we get 
$$
    \int_\R u (t, x) \ln u  (t, x) \ dx \ge 0,
$$
 which contradicts the second inequality in~\eqref{e:unozero}. This concludes the proof of the lemma. 
\end{proof}

\section{Acknowledgments}
The authors wish to thank Stefano Bianchini, Rinaldo Colombo, Elio Marconi, Mario Pulvirenti, Giuseppe Savar\'e and Giuseppe Toscani for valuable discussions. In particular, Lemma~\ref{l:ennio} is due to Elio Marconi and the equality~\eqref{e:savare} was pointed out to us by Giuseppe Savar\'e. GC is partially supported by the Swiss National Science Foundation
grant 200020\_156112 and by the ERC Starting Grant 676675 FLIRT. LVS is a member of the GNAMPA group of INDAM and of the PRIN National Project ``Hyperbolic Systems of Conservation Laws and Fluid Dynamics:
Analysis and Applications". Part of this work was done when LVS and MC were visiting the University of Basel, and its  kind hospitality is gratefully acknowledged.  
\bibliographystyle{plain}
\bibliography{singpar}
\end{document}